\documentclass{my-dcds-s}

\usepackage{amsmath}
\usepackage{paralist}
\usepackage{bbm}
\usepackage{booktabs}
\usepackage{graphics}
\usepackage{epsfig}
\usepackage{graphicx}
\usepackage{epstopdf}

\usepackage[ruled,lined]{algorithm2e}

\usepackage[colorlinks=true]{hyperref}
\hypersetup{urlcolor=blue, citecolor=red}

\def\currentvolume{11}
 \def\currentissue{6}
  \def\currentyear{2018}
   \def\currentmonth{December}
    \def\ppages{X--XX}
     \def\DOI{10.3934/dcdss.2018067}

\newtheorem{theorem}{Theorem}[section]
\theoremstyle{definition}
\newtheorem{remark}{Remark}

\newcommand{\eps}{\varepsilon}
\renewcommand{\phi}{\varphi}
\newcommand{\ds}{\, ds}
\newcommand{\du}{\, du}
\newcommand{\dint}{\displaystyle \int}
\newcommand{\R}{\mathbb{R}}

\DeclareMathOperator{\e}{e}
\DeclareMathOperator{\inte}{int}

\newcommand{\matj}{\mathcal{J}}

\newcommand{\nn}{\nonumber }

\newcommand{\parc}[1]{\ensuremath{\left(#1\right)}}

\title[Optimal control of non-autonomous SEIRS models]{Optimal control
of non-autonomous SEIRS models with vaccination and treatment}

\author[J. P. Mateus, P. Rebelo, S. Rosa, C. M. Silva and D. F. M. Torres]{}

\subjclass{Primary: 34H05, 92D30; Secondary: 37B55, 49M05.}

\keywords{Epidemic model, non-autonomous control system,
vaccination and treatment control, optimal control, numerical simulations.}

\email{jmateus@ipg.pt}
\email{rebelo@ubi.pt}
\email{rosa@ubi.pt}
\email{csilva@ubi.pt}
\email{delfim@ua.pt}

\thanks{Mateus was partially supported by FCT through
CMA-UBI (project UID/MAT/00212/2013), Rebelo by FCT
through CMA-UBI (project UID/MAT/00212/2013),
Rosa by FCT through IT (project UID/EEA/50008/2013),
Silva by FCT through CMA-UBI (project UID/MAT/00212/2013),
and Torres by FCT through CIDMA
(project UID/MAT/04106/2013) and TOCCATA (project
PTDC/EEI-AUT/2933/2014 funded by FEDER and COMPETE 2020).}

\thanks{$^*$ Corresponding author: delfim@ua.pt}

\begin{document}

\maketitle


\centerline{\scshape Joaquim P. Mateus}
\medskip
{\footnotesize
\centerline{Research Unit for Inland Development (UDI)}
\centerline{Polytechnic Institute of Guarda, 6300-559 Guarda, Portugal}
\medskip
\centerline{Centro de Matem\'atica e Aplica\c{c}\~oes da Universidade da Beira Interior (CMA-UBI)}
\centerline{Departamento de Matem\'atica}
 \centerline{Universidade da Beira Interior, 6201-001 Covilh\~a, Portugal}}

\medskip

\centerline{\scshape Paulo Rebelo}
\medskip
{\footnotesize
\centerline{Centro de Matem\'atica e Aplica\c{c}\~oes da Universidade da Beira Interior (CMA-UBI)}
\centerline{Departamento de Matem\'atica}
\centerline{ Universidade da Beira Interior, 6201-001 Covilh\~a, Portugal}
}

\medskip

\centerline{\scshape Silv\'erio Rosa}
\medskip
{\footnotesize
\centerline{Departamento de Matem\'atica and Instituto de Telecomunica\c{c}\~oes (IT)}
\centerline{Universidade da Beira Interior, 6201-001 Covilh\~a, Portugal}
}

\medskip

\centerline{\scshape C\'esar M. Silva}
\medskip
{\footnotesize
\centerline{Centro de Matem\'atica e Aplica\c{c}\~oes da Universidade da Beira Interior (CMA-UBI)}
\centerline{Departamento de Matem\'atica}
\centerline{ Universidade da Beira Interior, 6201-001 Covilh\~a, Portugal}
}

\medskip

\centerline{\scshape Delfim F. M. Torres$^*$}
\medskip
{\footnotesize
\centerline{Center for Research and Development in Mathematics and Applications (CIDMA)}
\centerline{Department of Mathematics, University of Aveiro, 3810-193 Aveiro, Portugal}
}


\begin{abstract}
We study an optimal control problem for a non-autonomous SEIRS model
with incidence given by a general function of the infective,
the susceptible and the total population, and with vaccination
and treatment as control variables. We prove existence and uniqueness
results for our problem and, for the case of mass-action incidence,
we present some simulation results designed to compare an autonomous
and corresponding periodic model, as well as the controlled
versus uncontrolled models.
\end{abstract}


\section{Introduction}

Frequently, decision makers must balance the effort put in treatment and vaccination
in order to give the best response to outbreaks of infectious diseases. Optimal control
is an important mathematical tool that can be used to find the best strategies
\cite{MyID:356,MyID:340,MyID:353}. Real-world problems, under recent
investigation with optimal control techniques, include Ebola
\cite{MyID:364} and Zika \cite{Ding:Tao:Zhu}.

In~\cite{Gaff-Schaefer-MBE-2009}, Gaff and Schaefer studied three distinct autonomous
epidemic models, having vaccination and treatment as control variables. They established
existence and, in some small interval, uniqueness of solutions
for the optimal control problems. In that paper, one of the main questions under study
is to know if the underlying epidemic structure has a significant impact on the obtained
optimal control strategy. One of the models discussed in~\cite{Gaff-Schaefer-MBE-2009}
was the SEIR, which is one of the most studied models in epidemiology. In the family
of SEIRS models, it is assumed that the population is divided in four compartments:
additionally to the infected class $I$, the susceptible class $S$ and the recovered class $R$,
present in SIR models, an exposed class $E$ is also considered, in order to divide
the infected population in the group of individuals that are infected and can infect others
(the infective class) and the individuals that are infected but are not yet able
to infect other individuals (the exposed or latent class). Such division of the population
is particularity suitable to include several infectious diseases, e.g. measles and,
assuming vertical transmission, rubella~\cite{Li-Smith-Wang-SJAM-2001}.
When there is no recovery, the model can be also used to describe diseases such as
Chagas' disease~\cite{Driessche-ME-2008}. It is also appropriate to model
hepatitis B and AIDS~\cite{Li-Smith-Wang-SJAM-2001}, and Ebola \cite{MyID:340}.
Although influenza can be modeled by a SEIRS model~\cite{Cori-Valleron-E-2012},
due to the short latency period, it is sometimes better to use the simpler
SIRS formulation~\cite{Edlund-E-2011}.

In~\cite{Gaff-Schaefer-MBE-2009} the parameters of the considered models
are assumed to be independent of time. This is not very realistic
in many situations, in particular due to periodic seasonal fluctuations.
A classical example of seasonal patterns of incidence, exhibited by some infectious diseases,
is given by data on weekly measles notification in England and Wales during the period
1948--1968~\cite{Anderson-May-OUP-1991}. Other examples occur in several childhood diseases
such as mumps, chicken-pox, rubella and pertussis~\cite{Martcheva-JBD-2009}. It is also worth
to mention that some environmental and demographic effects can be non-periodic. For example,
for some diseases like cholera and yellow fever, it is known that the size of the latency
period may decrease with global warming~\cite{Shope-EHP-1991}. In this work, we consider
a controlled SEIRS model with vaccination and treatment as control variables but,
unlike \cite{Gaff-Schaefer-MBE-2009}, we let the parameters of our model to be time dependent.
One of our main objectives is to discuss the effect of seasonal behavior on the optimal strategy.

Another important aspect of our work is that we consider a model with a general incidence
function given by some function $\phi$ of the susceptible, the infective and the total population.
This allows us to prove simultaneously results on existence and uniqueness of optimal solution
for models with several different incidence functions that have already been considered
in the context of SEIR/SEIRS models. In particular, our setting includes not only
Michaelis--Menten incidence functions, considered for instance in
\cite{Bai-Zhou-NARWA-2012, Gao-Chen-Teng-NARWA-2008,Kuniya-Nakata-AMC-2012,%
Nakata-Kuniya-JMAA-2010,Wang-Derrick-BIMAS-1978,Zhang-Teng-BMB-2007},
but also incidence functions that are not bilinear, which are appropriate to include
saturation effects as well as other non-linear
behaviors~\cite{Liu-Hethcote-Levin-JMB-1987,Zhou-Xiao-Li-CSF-2007}.
Additionally, and in contrast with~\cite{Gaff-Schaefer-MBE-2009},
we assume that immunity can be partial and thus a fraction of the
recovered individuals return to the susceptible class.

The paper is organized in the following way.
In Section~\ref{sec:model}, we describe our non-autonomous SEIRS model,
including treatment and vaccination as control variables,
and we formulate the optimal control problem under investigation.
Then, in Section~\ref{section:Existence-control}, we discuss the question
of existence of optimal solutions. The optimal controls are characterized
in Section~\ref{section:characterization-control} with
the help of Pontryagin's minimum principle and,
in Section~\ref{section:Uniqueness-control},
we present a result concerning uniqueness of the optimal control.
We end with Section~\ref{section:Simulation-control} of numerical simulations.

\vspace{-10pt}
\section{The non-autonomous SEIRS model and the optimal control problem}
\label{sec:model}

In practice, evolution on the number of susceptible, exposed,
infective and recovered, depends on some factors that can be controlled.
Two of the main factors are: treatment of infective
and vaccination of susceptible. For this reason,
we consider a non-autonomous SEIRS model including treatment,
denoted by $\mathbbm{T}$, and vaccination, denoted by $\mathbbm{V}$,
as control variables. Namely, we consider the optimal control problem
\begin{equation}
\label{seir003}
\tag{P}
\begin{gathered}
\matj\parc{I,\mathbbm{T},\mathbbm{V}} =\displaystyle \int_0^{t_f}
\left(\kappa_1 I(t)+\kappa_2 \mathbbm{T}^2(t)+\kappa_3\mathbbm{V}^2(t)\right) dt
\longrightarrow \min,\\
\begin{cases}
S'(t) = \Lambda\parc{t} -\beta\parc{t}\varphi(S(t),N(t),I(t))
-\mu\parc{t} S(t)+\eta\parc{t} R(t) - \mathbbm{V}(t)S(t),\\
E'(t) =  \beta\parc{t} \varphi(S(t),N(t),I(t)) -\parc{\mu \parc{t} + \eps \parc{t}} E(t),\\
I'(t) =  \eps\parc{t}E(t) -\parc{\mu\parc{t}+\gamma \parc{t}} I(t) -\mathbbm{T}(t)I(t),\\
R'(t) = \gamma \parc{t} I(t) -(\mu\parc{t}+\eta\parc{t}) R(t)
+ \mathbbm{T}(t)I(t) + \mathbbm{V}(t)S(t),\\
\end{cases}\\
(S(0),E(0),I(0),R(0))=(S_0,E_0,I_0,R_0),
\end{gathered}
\end{equation}
where $\kappa_1,\kappa_2,\kappa_3$
and $S_0,E_0,I_0,R_0$ are non-negative,
the state variables are absolutely continuous functions,
i.e., $(S(\cdot),E(\cdot),I(\cdot),R(\cdot))
\in AC\left([0,t_f];\R^4\right)$,
and the controls are Lebesgue integrable, i.e.,
$(\mathbbm{T}(\cdot),\mathbbm{V}(\cdot))
\in L^1\left([0,t_f]; [0,\tau_{\max}]\times[0,\nu_{\max}]\right)$.
Moreover, the following conditions hold:
\begin{enumerate}[$C$1)]
\item \label{cond-C1}
$\Lambda, \beta, \mu, \eta, \eps, \gamma \in C^1([0,t_f];\R)$ are $\omega$-periodic;

\item \label{cond-C2} $\varphi \in C^2(\R^3;\R)$;

\item \label{cond-C3} $\varphi(0,y,x)=\varphi(x,y,0)=0$
for all $x \in \R_0^+$ and $y \in \R^+$;

\item \label{cond-C4} function $f:\mathcal D \to \R$
defined in $\mathcal D=\{(x,y,z) \in [0,t_f]^3: x+z \le y\}$
by $f(x,y,z)=\varphi(x,y,z)/z$, if $z>0$, and
$f(x,y,0)=\lim_{z \to 0} \varphi(x,y,z)/z$,
is continuous and bounded.
\end{enumerate}

The total population $N(t)=S(t)+E(t)+I(t)+R(t)$,
$t \in [0, t_f]$, is not constant (see Remark~\ref{remark1}).
It can be also seen, from the equations that define our control
system, that we assume treatment to be applied to infective individuals only,
moving a fraction of them from the infected to the recovered compartment,
and that vaccination is applied to susceptible individuals only,
also moving a fraction of them to the recovered class.
Moreover, note that in our problem \eqref{seir003},
besides general non-autonomous parameters, we consider
a general incidence function. As examples of such incidence functions,
we mention the ones obtained by $\varphi(S,N,I)= SI/(1+\alpha I)$,
considered for instance  in~\cite{Safi-Garba-CMMM-2012,Zhang-Teng-CSF-2009},
$\varphi(S,N,I)= I^pS^q$, considered in
\cite{Hethcote-Lewis-Driessche-JMB-1989,Korobeinikov-Maini-MBE-2004,Liu-Hethcote-Levin-JMB-1987},
and $\varphi(S,N,I)=S I^p/(1+\alpha I^q)$, considered
in~\cite{Driessche-Hethcote-JMB-1991,Ruan-Wang-JDE-2003}.
Naturally, one needs to specify the incidence function,
the other parameters and functions of the model, in order to
carry out simulations and investigate particular situations.
This is done in Section~\ref{section:Simulation-control},
where we compare an autonomous problem with a corresponding periodic model,
as well as uncontrolled and corresponding controlled models.
Before that, we prove results on existence of optimal solution
(Theorem~\ref{mr:thm:exist}), necessary optimality conditions
(Theorem~\ref{thm:Pontryagin-SEIR-Mayer}), and uniqueness of solution
(Theorem~\ref{thm:uniq}) for \eqref{seir003}.


\section{Existence of optimal solutions}
\label{section:Existence-control}

Problem \eqref{seir003} is an optimal control problem in Lagrange form:
\begin{equation}
\label{general-problem}
\begin{gathered}
J(x,u)=\displaystyle \int_{t_0}^{t_1} \mathcal{L}(t,x(t),u(t)) \ dt
\longrightarrow \min,\\
\begin{cases}
x'(t)=f\left(t,x(t),u(t)\right), \quad \text{a.e.} \ t \in [t_0,t_1],\\
x(t_0)=x_0,
\end{cases}\\
x(\cdot) \in AC\left([t_0,t_1];\R^n\right),
\quad u(\cdot) \in L^1([t_0,t_1];U\subset \R^m).
\end{gathered}
\end{equation}
In the above context, we say that a pair $(x,u)
\in AC\left([t_0,t_1];\R^n\right) \times L^1([t_0,t_1];U)$
is feasible if it satisfies the Cauchy problem in~\eqref{general-problem}.
We denote the set of all feasible pairs by $\mathcal F$.
Next, we recall a theorem about existence of solution for optimal control problems
\eqref{general-problem}, contained in Theorem III.4.1 and Corollary III.4.1 
in~\cite{Fleming-Rishel-Springer-Verlag-1975}.

\begin{theorem}[See \cite{Fleming-Rishel-Springer-Verlag-1975}]
\label{teo:existence-Fleming-Raymond-1974}
For problem~\eqref{general-problem}, suppose that $f$ and $\mathcal{L}$
are continuous and there exist positive constants $C_1$ and $C_2$ such that,
for $t \in \R$, $x,x_1,x_2 \in \R^n$ and $u \in \R^m$, we have
\begin{enumerate}[a)]
\item \label{teo:existence-Fl-Ray-1974-a}
$\|f(t,x,u)\| \le C_1(1+\|x\|+\|u\|)$;

\item \label{teo:existence-Fl-Ray-1974-b}
$\|f(t,x_1,u)-f(t,x_2,u)\| \le C_2\|x_1-x_2\|(1+\|u\|)$;

\item \label{teo:existence-Fl-Ray-1974-c}
$\mathcal F$ is non-empty;

\item \label{teo:existence-Fl-Ray-1974-d}
$U$ is closed;

\item \label{teo:existence-Fl-Ray-1974-e}
there is a compact set $S$ such that $x(t_1) \in S$
for any state variable $x$;

\item \label{teo:existence-Fl-Ray-1974-f}
$U$ is convex, $f(t,x,u)=\alpha(t,x)+\beta(t,x)u$,
and $\mathcal{L}(t,x,\cdot)$ is convex on $U$;

\item \label{teo:existence-Fl-Ray-1974-g}
$\mathcal{L}(t,x,u)\ge c_1|u|^\beta-c_2$, for some $c_1>0$ and $\beta >1$.
\end{enumerate}
Then, there exist $(x^*,u^*)$ minimizing $J$ on $\mathcal F$.
\end{theorem}

Before applying Theorem~\ref{teo:existence-Fleming-Raymond-1974}
to obtain an existence result to our problem \eqref{seir003},
we make a useful remark.

\begin{remark}
\label{remark1}
Adding the equations in~\eqref{seir003}, we obtain that
$N'(t)=\Lambda(t)-\mu(t)N(t)$, which describes the behavior
of the total population $N(t)$. Since $N_0=N(t_0)=S_0+E_0+I_0+R_0$, we have
\begin{equation*}
N(t)=N_0\e^{\int_0^{t_f}\mu(s)\ds}
+\dint_0^{t_f} \Lambda(u)\e^{\int_u^{t_f}\mu(s)\ds} \du \le K,
\end{equation*}
where $K= N_0 +\sup_{t \in [0,t_f]} \Lambda(t) / \inf_{t \in [0,t_f]} \mu(t)$.
Thus,
\begin{equation}
\label{eq:bbbb1}
S(t),E(t),I(t),R(t) \le K.
\end{equation}
Consider the problem obtained by replacing function $\phi$ in~\eqref{seir003}
by some bounded and twice continuously differentiable function $\psi$ such that
$\psi(S,N,I)=\phi(S,N,I)$ for $(S,N,I) \in [0,K]^3$ (and maintaining the initial condition,
the cost functional and the set of admissible controls). By~\eqref{eq:bbbb1},
this new problem has exactly the same solutions as problem~\eqref{seir003}.
\end{remark}

\begin{theorem}[Existence of solutions for the optimal control problem \eqref{seir003}]
\label{mr:thm:exist}
There exists an optimal control pair $\left(\mathbbm{T}^*, \mathbbm{V}^*\right)$
and a corresponding quadruple $(S^*,E^*,I^*,\break R^*)$ of the initial
value problem in \eqref{seir003} that minimizes the cost functional
$\matj$ in~\eqref{seir003} over
$L^1\left([0,t_f]; [0,\tau_{\max}] \times [0,\nu_{\max}]\right)$.
\end{theorem}

\begin{proof}
We show that the problem obtained from \eqref{seir003}
by replacing the function $\phi$ by some of the functions $\psi$
in Remark~\ref{remark1} satisfies the conditions in
Theorem~\ref{teo:existence-Fleming-Raymond-1974}.
By Remark~\ref{remark1}, using conditions C\ref{cond-C1}) and C\ref{cond-C3}),
we immediately obtain~\ref{teo:existence-Fl-Ray-1974-a}) and~\ref{teo:existence-Fl-Ray-1974-b}).
Conditions \ref{teo:existence-Fl-Ray-1974-c}) and~\ref{teo:existence-Fl-Ray-1974-d})
are immediate from the definition of $\mathcal F$ and since $U=[0,\tau_{\max}]\times [0,\nu_{\max}]$.
By~\eqref{eq:bbbb1}, we conclude that all state variables are in the compact set
$$
\{(x,y,z,w) \in (\R_0^+)^4: 0 \le x+y+z+w\le K\}
$$
and condition~\ref{teo:existence-Fl-Ray-1974-e}) follows.
Since the state equations are linearly dependent on the controls,
we obtain~\ref{teo:existence-Fl-Ray-1974-f}). Finally, $\mathcal{L}$
is convex in the controls since it is quadratic. Moreover,
$\mathcal{L} = k_1 I+k_2 \mathbbm{T}^2+k_3\mathbbm{V}^2
\ge \min\{k_2,k_3\}(\mathbbm{T}^2+\mathbbm{V}^2)
\ge \min\{k_2,k_3\}\|(\mathbbm{T},\mathbbm{V})\|^2$
and we establish~\ref{teo:existence-Fl-Ray-1974-g})
with $c_1=\min\{k_2,k_3\}$. Thus, taking into account
Remark~\ref{remark1}, the result follows from
Theorem~\ref{teo:existence-Fleming-Raymond-1974}.
\end{proof}

\vspace{-6pt}
\section{Characterization of the optimal controls}
\label{section:characterization-control}

Now we address the question of how to identify
the solutions predicted by Theorem~\ref{mr:thm:exist}.
We do this with the help of the celebrated Pontryagin
Maximum Principle \cite{book:Pont}. This is possible
because all control minimizers $\mathbbm{T}^*$ and $\mathbbm{V}^*$
of problem \eqref{seir003} are in fact essentially bounded. Indeed,
Theorem~\ref{teo:existence-Fleming-Raymond-1974}
requires $U$ to be a closed set, not necessarily bounded,
and it may happen, in general, that the optimal controls predicted by
Theorem~\ref{teo:existence-Fleming-Raymond-1974} are in $L^1$
but not in $L^\infty$ and do not satisfy the
Pontryagin Maximum Principle \cite{MyID:015}. However,
in our case, $U$ is compact and we can conclude that
the $L^1$ optimal controls $\left(\mathbbm{T}^*, \mathbbm{V}^*\right)$
of problem \eqref{seir003}, predicted by Theorem~\ref{mr:thm:exist},
are in fact in $L^\infty$, as required by the necessary optimality
conditions \cite{book:Pont}. Moreover, our optimal control
problem \eqref{seir003} has only given initial conditions,
with the state variables being free at the terminal time, that is,
$S(t_f)$, $E(t_f)$, $I(t_f)$, $R(t_f)$ are free. This implies
that abnormal minimizers \cite{abnormal} are not possible in our context
and we can fix the cost multiplier associated with
the Lagrangian $\mathcal{L}$ to be one: for our optimal
control problem \eqref{seir003}, the Hamiltonian is given by
\begin{equation}
\label{ham}
\begin{split}
\mathcal H&(t,(S,E,I,R),(p_1,p_2,p_3,p_4),(\mathbbm{T},\mathbbm{V}))
= \kappa_1 I+\kappa_2 \mathbbm{T}^2 + \kappa_3\mathbbm{V}^2\\[1mm]
&+p_1\left[\Lambda\parc{t} -\beta\parc{t}\varphi(S,N,I) -\mu\parc{t} S+\eta\parc{t} R-\mathbbm{V}S \right] \\[1mm]
& + p_2\left[\beta\parc{t} \varphi(S,N,I) -\parc{\mu \parc{t} + \eps \parc{t}} E\right] \nn \\[1mm]
& + p_3\left[\eps\parc{t}E -\parc{\mu\parc{t} +\gamma \parc{t}} I  -\mathbbm{T}I\right] \nn \\[1mm]
& + p_4\left[\gamma \parc{t} I -\mu\parc{t}R -\eta\parc{t} R + \mathbbm{T}I+\mathbbm{V}S\right].
\end{split}
\end{equation}

In what follows, we use the operator $\partial_i$ to denote
the partial derivative with respect to the $i$th variable.

\begin{theorem}[Necessary optimality conditions for the optimal control problem \eqref{seir003}]
\label{thm:Pontryagin-SEIR-Mayer}
If $((S^*,E^*,I^*,R^*),(\mathbbm{T}^*,\mathbbm{V}^*))$ is a minimizer of problem~\eqref{seir003},
then there exist multipliers $(p_1(\cdot),p_2(\cdot),p_3(\cdot),p_4(\cdot))
\in AC([0,t_f];\R^4)$ such that
\begin{equation}
\label{eq:Pontryagin-SEIR-Mayer-1}
\begin{cases}
p_1'= (p_1-p_2)\beta\parc{t}\parc{{\partial_1 \varphi}\parc{S,N,I} 
+ {\partial_2 \varphi}\parc{S,N,I}} + p_1 \parc{\mu\parc{t}+\mathbbm{V}} -p_4\mathbbm{V},\\[1mm]
p_2' = (p_1-p_2)\beta\parc{t}{\partial_2 \varphi}\parc{S,N,I}
+p_2\parc{\mu\parc{t}+\eps\parc{t}}-p_3\eps\parc{t},\\[1mm]
p_3' = p_3\parc{\mu\parc{t}+\gamma \parc{t}+\mathbbm{T}}+\parc{p_1-p_2}\beta \parc{t}
\parc{{\partial_2 \varphi}\parc{S,N,I} + {\partial_3 \varphi}\parc{S,N,I}}\\[1mm]
\quad \quad -p_4 \parc{\gamma \parc{t}+\mathbbm{T}}-\kappa_1,\\[1mm]
p_4' = (p_1-p_2)\beta\parc{t}{\partial_2 \varphi}\parc{S,N,I}
+\mu\parc{t} p_4-\eta\parc{t} p_1 +\eta\parc{t} p_4,
\end{cases}
\end{equation}
for almost all $t \in [0,t_f]$, with transversality conditions
\begin{equation}
\label{eq:Pontryagin-SEIR-Mayer-5}
p_1(t_f)=p_2(t_f)=p_3(t_f)=p_4(t_f)=0.
\end{equation}
Furthermore, the optimal control pair is given by

\ \vspace{-10pt}
\begin{equation}
\label{eq:Pontryagin-SEIR-Mayer-6}
\mathbbm{T}^*=\min\left\{\max\left\{0,\frac{I^*(p_3-p_4)}{2k_2}\right\},\tau_{\max}\right\}
\end{equation}
and
\begin{equation}
\label{eq:Pontryagin-SEIR-Mayer-7}
\mathbbm{V}^*=\min\left\{\max\left\{0,\frac{S^*(p_1-p_4)}{2k_3}\right\},v_{\max}\right\}.
\end{equation}
\end{theorem}

\begin{proof}
Direct computations show that equations~\eqref{eq:Pontryagin-SEIR-Mayer-1}
are a consequence of the adjoint system of the Pontryagin Minimum Principle (PMP)
\cite{book:Pont}. Similarly, \eqref{eq:Pontryagin-SEIR-Mayer-5} are
directly given by the transversality conditions of the PMP.
It remains to characterize the controls using the minimality condition
of the PMP \cite{book:Pont}. The minimality condition on the set
$\left\{t \in [0,t_f]: 0<\mathbbm{V}^*(t) < \nu_{\max}
\ \text{ and } \ 0<\mathbbm{T}^*(t) < \tau_{\max} \right\}$
is
$$
\frac{\partial \mathcal H}{\partial \mathbbm{V}^*}
=-p_1S+p_4S+2k_3\mathbbm{V}^* = 0\quad \text{and}
\quad \frac{\partial \mathcal H}{\partial \mathbbm{T}^*}
=-p_3I+p_4I+2k_2\mathbbm{T}^*  = 0
$$
and thus, on this set,
$$
\mathbbm{V}^*=\frac{(p_1-p_4)S}{2k_3}
\quad \text{and} \quad
\mathbbm{T}^*=\frac{(p_3-p_4)I}{2k_2}.
$$
If $t \in \inte \{t \in [t_0,t_1]: \mathbbm{V}^*(t) = \nu_{\max}\}$,
then the minimality condition is
$$
\frac{\partial \mathcal H}{\partial \mathbbm{V}^*}
= -p_1S+p_4S+2k_3\mathbbm{V}^*\leq 0 \quad
\Leftrightarrow \quad \frac{(p_1-p_4)S}{2k_3}\ge \nu_{\max}.
$$
Analogously, if $t \in \inte \{t \in [t_0,t_1]:
\mathbbm{T}^*(t) = \tau_{\max}\}$,
then the minimality condition is
$$
\frac{\partial \mathcal H}{\partial \mathbbm{T}^*}
= -p_3I+p_4I+2k_2\mathbbm{T}^*\leq 0\quad
\Leftrightarrow \quad \frac{(p_3-p_4)I}{2k_2}\ge \tau_{\max}.
$$
If $t \in \inte \{t \in [t_0,t_1]: \mathbbm{V}^*(t) = 0\}$,
then the minimality condition is
$$
\frac{\partial \mathcal H}{\partial \mathbbm{V}^*}
= -p_1S+p_4S+2k_3\mathbbm{V}^*\geq 0 \quad
\Leftrightarrow \quad \frac{(p_1-p_4)S}{2k_3}\le 0.
$$
Analogously, if $t \in \inte \{t \in [t_0,t_1]: \mathbbm{T}^*(t) = 0\}$,
then the minimality condition is
$$
\frac{\partial \mathcal H}{\partial \mathbbm{T}^*}
= -p_3I+p_4I+2k_2\mathbbm{T}^*\geq 0\quad
\Leftrightarrow \quad \frac{(p_3-p_4)I}{2k_2}\le 0.
$$
Therefore, we obtain~\eqref{eq:Pontryagin-SEIR-Mayer-6}
and~\eqref{eq:Pontryagin-SEIR-Mayer-7}.
\end{proof}

\vspace{-10pt}
\section{Uniqueness of the optimal control}
\label{section:Uniqueness-control}

In this section we show that the optimal solution of~\eqref{seir003} is unique.
The result is new even in the particular case where all the parameters of the model
are time-invariant (autonomous case), extending the local uniqueness
results in~\cite{fister1998optimizing,Gaff-Schaefer-MBE-2009},
which are valid only in a small time interval, to uniqueness in a global sense,
that is, uniqueness of solution of \eqref{seir003} along all the time interval $[0, t_f]$
where the optimal control problem is defined. The proof of Theorem~\ref{thm:uniq}
is nontrivial, lengthy and technical. It can be summarized as follows:
\begin{enumerate}
\item By contradiction, we assume that there are two distinct optimal
pairs of state and co-state variables $\xi = (S,E,I,R,p_1,p_2,p_3,p_4)$ and
$\xi^* = (S^*,E^*,I^*,R^*,$ $p_1^*,p_2^*,p_3^*,p_4^*)$,
which correspond to two different optimal controls
$u=\left(\mathbbm{T}, \mathbbm{V}\right)$ and
$u^*=\left(\mathbbm{T}^*, \mathbbm{V}^*\right)$,
in agreement with \eqref{eq:Pontryagin-SEIR-Mayer-6} and
\eqref{eq:Pontryagin-SEIR-Mayer-7}.

\item We show that both minimizing trajectories
are in a positively invariant region $\Gamma$,
which is given by the total population and is
independent on the controls.

\item We make a change of variable and prove that we have a contradiction unless
$\xi = \xi^*$ in a small time interval $[0, T]$; from
\eqref{eq:Pontryagin-SEIR-Mayer-6}--\eqref{eq:Pontryagin-SEIR-Mayer-7},
$u = u^*$ in $[0, T]$.

\item If $[0,T]$ contains the time range of
the optimal control problem \eqref{seir003}, then we are done.
Otherwise, taking for initial conditions at time $T$ the values
of the state trajectories at the right-end of the interval $[0, T]$,
we obtain uniqueness for the interval $[T, 2T]$, because the estimates that
allow us to obtain $T$ are only related with the maximum value of the parameters
and the bounds for the state and co-state variables on the invariant
region $\Gamma$ of the new control problem, and are therefore the same
as the ones already obtained for $[0,T]$

\item Iterating the procedure, we obtain uniqueness throughout the interval $[0, t_f]$
after a finite number of steps.
\end{enumerate}

\begin{theorem}[Uniqueness of solution for the optimal control problem \eqref{seir003}]
\label{thm:uniq}
The solution of the optimal control problem \eqref{seir003} is unique.
\end{theorem}

\begin{proof}
Let us assume that we have two optimal solutions corresponding
to state trajectories and adjoint variables $(S,E,I,R)$ and $(p_1,p_2,p_3,p_4)$
and $(\bar S,\bar E,\bar I,\bar R)$ and $(\bar p_1,\bar p_2,\bar p_3,\bar p_4)$.
Then, we show that the two are the same, at least in some small interval.
To achieve this, we make the change of variables
$$
S(t)=e^{\alpha t}s(t), \quad E(t)=e^{\alpha t}e(t),
\quad I(t)=e^{\alpha t}i(t), \quad R(t)=e^{\alpha t}r(t)
$$
and
$$
p_1(t)=e^{-\alpha t}\phi_1(t), \quad p_2(t)=e^{-\alpha t}\phi_2(t),
\quad p_3(t)=e^{-\alpha t}\phi_3(t), \quad p_4(t)=e^{-\alpha t}\phi_4(t).
$$
Naturally, setting $n(t)=s(t)+e(t)+i(t)+r(t)$, we have
$$
N(t)=S(t)+E(t)+I(t)+R(t)=\e^{\alpha t} n(t).
$$
By Proposition~1 in~\cite{Mateus-Silva-AMC-2014}, we can assume
that the trajectories lie in a compact and forward invariant set $\Gamma$,
which can be chosen independently of the control functions
$\mathbbm{T}$ and $\mathbbm{V}$. Using the differentiability
assumption~C\ref{cond-C2}), we get
\begin{equation}
\label{bound_phi}
\begin{split}
|\varphi(A,B,C)-\varphi&(\bar A,\bar B,\bar C)|\\
& \le |\varphi(A,B,C)-\varphi(\bar A,B,C)|+|\varphi(\bar A,B,C)-\varphi(\bar A,\bar B,C)|\\
& \quad +|\varphi(\bar A,\bar B,C)-\varphi(\bar A,\bar B,\bar C)|\\
& \le M_1^u|A-\bar A|+M_2^u|B-\bar B|+M_3^u|C-\bar C|,
\end{split}
\end{equation}
where, because $\Gamma$ is compact, we have
\begin{equation}
\label{eq:maj-partial-i}
M_i^u:=\sup_{x \in \Gamma}|\partial_i \varphi|<+\infty, \quad i=1,2,3.
\end{equation}
Considering the first equation in~\eqref{seir003}, we get, a.e. in $t \in [0,t_f]$, that
$$
\alpha e^{\alpha t}s+e^{\alpha t} \dot{s}
=\Lambda -\beta\varphi(e^{\alpha t}s, e^{\alpha t}n,
e^{\alpha t}i) -\mu e^{\alpha t}s
+ \eta e^{\alpha t}r - \mathbbm{V}e^{\alpha t}s
$$
and thus, for a.a. $t \in [0,t_f]$,
$$
\alpha s + \dot{s}=\frac {\Lambda}{e^{\alpha t}}
-\frac {\beta}{e^{\alpha t}}\varphi(e^{\alpha t}s,
e^{\alpha t}n, e^{\alpha t}i) -\mu s
+ \eta r - \mathbbm{V} s.
$$
Subtracting from the above equation the corresponding barred equation, we obtain
\[
\begin{split}
\alpha (s-\bar{s}) + (\dot{s}-\dot{\bar {s}})
& = -\frac {\beta}{e^{\alpha t}}(\varphi(e^{\alpha t}s, e^{\alpha t}n, e^{\alpha t}i)
-\varphi(e^{\alpha t}\bar{s}, e^{\alpha t}\bar{n}, e^{\alpha t}\bar{i})) -\mu (s-\bar{s})\\
& \quad + \eta (r-\bar{r}) - (\mathbbm{V} s-\bar{\mathbbm{V}}\bar{s}).
\end{split}
\]
Multiplying by $(s-\bar{s})$, integrating from $0$ to $T$,
and noting that $s(0)=\bar{s}(0)$,
\begin{equation}
\label{eq:s-bar-s...}
\begin{split}
&\frac{1}{2}(s(T)-\bar {s}(T))^2 + \alpha \int_0^T (s-\bar{s})^2 dt \\
& = -\int_0^T \frac {\beta}{e^{\alpha t}}(s-\bar{s})(\varphi(e^{\alpha t}s,
e^{\alpha t}n, e^{\alpha t}i) - \varphi(e^{\alpha t}\bar{s},
e^{\alpha t}\bar{n}, e^{\alpha t}\bar{i}))dt\\
& \quad -\int_0^T \mu (s-\bar{s})^2 dt
+ \int_0^T \eta (s-\bar{s})(r-\bar{r}) dt
- \int_0^T  (\mathbbm{V} s-\bar {\mathbbm{V}} \bar{s})(s-\bar{s}) dt
\end{split}
\end{equation}
and, by~\eqref{bound_phi}, we obtain
\begin{equation}
\label{eq:s-bar-s...2}
\begin{split}
&\frac{1}{2}(s(T)-\bar {s}(T))^2 + \alpha \int_0^T (s-\bar{s})^2 dt \\
& \le \int_0^T \frac{\beta}{e^{\alpha t}} |s-\bar s|
(M_1^u|e^{\alpha t}s-e^{\alpha t}\bar s|+M_2^u|e^{\alpha t}n-e^{\alpha t}\bar n|
+M_3^u|e^{\alpha t}i-e^{\alpha t}\bar{i}|)dt\\
& \quad -\int_0^T \mu (s-\bar{s})^2 dt
+ \int_0^T \eta (s-\bar{s})(r-\bar{r}) dt
- \int_0^T  (\mathbbm{V} s-\bar {\mathbbm{V}} \bar{s})(s-\bar{s}) dt\\
&  = \int_0^T \beta M_1^u(s-\bar s)^2 dt+\int_0^T \beta M_2^u|s
-\bar s||n-\bar n| dt+\int_0^T \beta M_3^u|s-\bar s||i-\bar{i}|dt\\
& \quad -\int_0^T \mu (s-\bar{s})^2 dt
+ \int_0^T \eta (s-\bar{s})(r-\bar{r}) dt
- \int_0^T  (\mathbbm{V} s-\bar {\mathbbm{V}} \bar{s})(s-\bar{s}) dt\\
& \le C_1 \int_0^T (s-\bar s)^2+(i-\bar{i})^2+(e-\bar{e})^2+(r-\bar{r})^2 dt
+ K_1\e^{\alpha T}\int_0^T  (\mathbbm{V}-\bar {\mathbbm{V}})^2 dt,
\end{split}
\end{equation}
where $K_1$ depends on the bounds for $\bar S$ and $\mathbbm{V}$
and $C_1=\beta^u M_1^u+2\beta^u M_2^u+\beta^u M_3^u+\eta^u+2K_1$
(recall that $M_i^u$ is given by~\eqref{eq:maj-partial-i}).
We now use some estimates for $(\mathbbm{V}-\bar {\mathbbm{V}})^2$
and $(\mathbbm{T}-\bar {\mathbbm{T}})^2$ that we prove later.
Namely, we have
\begin{equation}
\label{eq:bound-V-Vbar}
\left(\mathbbm{V}-\bar {\mathbbm{V}}\right)^2
\le C_9\e^{2\alpha T}
\left[(s-\bar s)^2+(\phi_1-\bar{\phi}_1)^2+(\phi_4-\bar{\phi}_4)^2\right],
\end{equation}
where $C_9$ depends on $\alpha$ and on the maximum of $S$,
$\bar p_1$ and $\bar p_4$ on $\Gamma$, and
\begin{equation}
\label{eq:bound-T-Tbar}
\left(\mathbbm{T}-\bar {\mathbbm{T}}\right)^2
\le C_{10}\e^{2\alpha T}
\left[(i-\bar i)^2+(\phi_3-\bar{\phi}_3)^2+(\phi_4-\bar{\phi}_4)^2\right],
\end{equation}
where $C_{10}$ depends on $\alpha$ and on the maximum 
of $I$, $\bar p_3$ and $\bar p_4$ on $\Gamma$
(see~\eqref{majV} and~\eqref{majT}).
Define
$$
\Psi(t)=(s(t)-\bar s(t))^2+(e(t)-\bar e(t))^2+(i(t)-\bar i(t))^2+(r(t)-\bar r(t))^2
$$
and
$$
\Phi(t)=(\phi_1(t)-\bar \phi_1(t))^2+(\phi_2(t)
-\bar \phi_2(t))^2+(\phi_3(t)-\bar \phi_3(t))^2
+(\phi_4(t)-\bar \phi_4(t))^2.
$$
Observe that $\Psi(t)\ge0$ and $\Phi(t)\ge0$ for all $t$.
By~\eqref{eq:bound-V-Vbar}, we obtain
\begin{equation*}
\begin{split}
\frac{1}{2}&(s(T)-\bar {s}(T))^2 + \alpha \int_0^T (s-\bar{s})^2 dt \\
& \le C_1 \int_0^T (s-\bar s)^2+(i-\bar{i})^2+(e-\bar{e})^2+(r-\bar{r})^2 dt\\
& \quad + K_1C_9\int_0^T (s-\bar s)^2 +(\phi_1-\bar \phi_1)^2+(\phi_4-\bar \phi_4)^2 dt
\end{split}
\end{equation*}
\begin{equation}
\label{eq:s-bar-s}
\begin{split}
& \le (C_1+K_1C_9) \int_0^T (s-\bar s)^2+(i-\bar{i})^2+(e-\bar{e})^2+(r-\bar{r})^2 \\
& \quad +(\phi_1-\bar \phi_1)^2+(\phi_4-\bar \phi_4)^2 dt\\
& \le \left(C_1+K_1C_9\e^{3\alpha T}\right) \int_0^T \Phi(t)+\Psi(t) dt.
\end{split}
\end{equation}
By similar computations to the ones in~\eqref{eq:s-bar-s...}
and~\eqref{eq:s-bar-s...2}, and using the inequality $xy\le x^2+y^2$, we get
\begin{equation}
\label{eq:e-bar-e}
\frac{1}{2}(e(T)-\bar {e}(T))^2 + \alpha \int_0^T (e-\bar{e})^2 dt
\le C_2 \int_0^T \Phi(t)+\Psi(t) dt,
\end{equation}
where $C_2=\beta^u M_1^u+2\beta^u M_2^u+\beta^u M_3^u$. Recalling that
$$
(xy-\bar x\bar y)(w-\bar w) \le C ((x-\bar x)^2+(y-\bar y)^2+(w-\bar w)^2),
$$
where $C>0$ depends on the bounds for $\bar x$ and $y$,
from the third equation in~\eqref{seir003}, by similar computations
to the ones in~\eqref{eq:s-bar-s...} and~\eqref{eq:s-bar-s...2},
and using~\eqref{eq:bound-T-Tbar}, we conclude that
\begin{equation}
\label{eq:i-bar-i}
\frac{1}{2}(i(T)-\bar {i}(T))^2 + \alpha \int_0^T (i-\bar{i})^2 dt
\le \left(C_3+K_2\left(C_{10}+2\right)\e^{3\alpha T}\right)
\int_0^T \Phi(t)+\Psi(t) dt,
\end{equation}
where $K_2$ depends on the maximum of $\bar I$
and $\mathbbm{T}$ on $\Gamma$ and $C_3=\eps^u$.
From the fourth equation in~\eqref{seir003}, by similar computations
to the ones in~\eqref{eq:s-bar-s...} and~\eqref{eq:s-bar-s...2},
and using~\eqref{eq:bound-V-Vbar} and~\eqref{eq:bound-T-Tbar},
we conclude that
\begin{multline}
\label{eq:r-bar-r}
\frac{1}{2}\left(r(T)-\bar {r}(T)\right)^2 + \alpha \int_0^T (r-\bar{r})^2 dt\\
\le \left[C_4+(K_3(C_{10}+1)\e^{3\alpha T}+K_4(C_9+1))\e^{3\alpha T}\right]
\int_0^T \Phi(t)+\Psi(t) dt,
\end{multline}
where $K_3$ and $K_4$ depend on the bounds for $\bar I$, $\bar S$,
$\mathbbm{T}$ and $\mathbbm{V}$ and $C_4=\gamma^u$. From the first
equation of \eqref{eq:Pontryagin-SEIR-Mayer-1}, we get that
\[
\begin{split}
&-\alpha e^{-\alpha t}\phi_1+e^{-\alpha t} \dot{\phi}_1\\
& = e^{-\alpha t}(\phi_1-\phi_2)\beta \parc{{\partial_1 \varphi}
\parc{e^{\alpha t}s,e^{\alpha t}n,e^{\alpha t}i}
+ {\partial_2 \varphi}\parc{e^{\alpha t}s,e^{\alpha t}n,e^{\alpha t}i}}\\
& \quad + e^{-\alpha t}\phi_1\parc{\mu+\mathbbm{V}}-e^{-\alpha t}\phi_4\mathbbm{V}
\end{split}
\]
for a.a. $t \in [0,t_f]$ and thus
\begin{equation}
\label{eqp1}
\begin{split}
-\alpha \phi_1 + \dot{\phi}_1
& = (\phi_1-\phi_2)\beta \parc{{\partial_1 \varphi}\parc{e^{\alpha t}s,
e^{\alpha t}n,e^{\alpha t}i} + {\partial_2 \varphi}\parc{e^{\alpha t}s,
e^{\alpha t}n,e^{\alpha t}i}} \\
& \quad + \phi_1\parc{\mu+\mathbbm{V}}-\phi_4\mathbbm{V}
\end{split}
\end{equation}
for a.a. $t \in [0,t_f]$. Using~C\ref{cond-C2}), we get
\begin{align*}
& |\phi_j {\partial_i \varphi}\parc{e^{\alpha t}s,e^{\alpha t}n,
e^{\alpha t}i} - \bar{\phi}_j {\partial_i \varphi}\parc{e^{\alpha t}\bar{s},
e^{\alpha t}\bar{n},e^{\alpha t}\bar{i}}| \\
& \le |\phi_j {\partial_i \varphi}\parc{e^{\alpha t}s,
e^{\alpha t}n,e^{\alpha t}i} - \phi_j {\partial_i \varphi}\parc{e^{\alpha t}\bar{s},
e^{\alpha t}\bar{n},e^{\alpha t}\bar{i}}|\\
&\quad + |\phi_j {\partial_i \varphi}\parc{e^{\alpha t}\bar{s},
e^{\alpha t}\bar{n},e^{\alpha t}\bar{i}}- \bar{\phi}_j
{\partial_i \varphi}\parc{e^{\alpha t}\bar{s},
e^{\alpha t}\bar{n},e^{\alpha t}\bar{i}}| \\
&= \phi_j |{\partial_i \varphi}\parc{e^{\alpha t}s,
e^{\alpha t}n,e^{\alpha t}i} - {\partial_i \varphi}\parc{e^{\alpha t}\bar{s},
e^{\alpha t}\bar{n},e^{\alpha t}\bar{i}}|\\&\quad + |\phi_j
- \bar{\phi}_j| |\partial_i \varphi \parc{e^{\alpha t}\bar{s},
e^{\alpha t}\bar{n},e^{\alpha t}\bar{i}}|\\
&= \phi_j \left [|\partial_i \varphi\parc{e^{\alpha t}s,
e^{\alpha t}n,e^{\alpha t}i} - {\partial_i \varphi}\parc{e^{\alpha t}\bar{s},
e^{\alpha t}n,e^{\alpha t}i}|+|\partial_i \varphi\parc{e^{\alpha t}\bar{s},
e^{\alpha t}n,e^{\alpha t}i} \right. \\
&\quad \left. - {\partial_i \varphi}\parc{e^{\alpha t}\bar{s},
e^{\alpha t}\bar{n},e^{\alpha t}i}|+|\partial_i \varphi\parc{e^{\alpha t}\bar{s},
e^{\alpha t}\bar{n},e^{\alpha t}i}-\partial_i \varphi\parc{e^{\alpha t}\bar{s},
e^{\alpha t}\bar{n},e^{\alpha t}\bar{i}}| \right ]\\
&\quad + |\phi_j- \bar{\phi}_j| |\partial_i \varphi\parc{e^{\alpha t}\bar{s},
e^{\alpha t}\bar{n},e^{\alpha t}\bar{i}}|
\end{align*}
\[
\begin{split}
&\le \phi_j [M_{i1}^u |e^{\alpha t}s-e^{\alpha t}\bar{s}|+M_{i2}^u|e^{\alpha t}n
-e^{\alpha t}\bar{n}|+M_{i3}^u|e^{\alpha t}i-e^{\alpha t}\bar {i}|]\\
& \quad + |\phi_j- \bar{\phi}_j| |\partial_i \varphi\parc{e^{\alpha t}\bar{s},
e^{\alpha t}\bar{n},e^{\alpha t}\bar{i}}|\\
& \le \phi_j^u [M_{i1}^u e^{\alpha t}|s-\bar{s}|+M_{i2}^u e^{\alpha t}|n
-\bar{n}|+M_{i3}^u e^{\alpha t}|i-\bar {i}|] + M_i^u |\phi_j- \bar{\phi}_j|,
\end{split}
\]
where, by C\ref{cond-C1}) and since $\Gamma$ is compact, we have
\begin{equation}
\label{eq:maj-partial-ij}
M_{ij}^u:=\sup_{x \in \Gamma} \left|\partial_j\partial_i\varphi(x)\right|
< +\infty, \quad i,j \in \{1,2,3\}.
\end{equation}
Subtracting from equation~\eqref{eqp1} the corresponding barred equation,
we get
\[
\begin{split}
&-\alpha \phi_1 + \dot{\phi}_1 +\alpha \bar{\phi}_1 - \dot{\bar{\phi}}_1\\
&= \beta(\phi_1-\phi_2) \parc{{\partial_1 \varphi}\parc{e^{\alpha t}s,
e^{\alpha t}n,e^{\alpha t}i} + {\partial_2 \varphi}\parc{e^{\alpha t}s,
e^{\alpha t}n,e^{\alpha t}i}} + \phi_1\parc{\mu+\mathbbm{V}}-\phi_4\mathbbm{V}\\
& -\beta(\bar\phi_1-\bar\phi_2) \parc{{\partial_1 \varphi}\parc{e^{\alpha t}\bar s,
e^{\alpha t}\bar n,e^{\alpha t}\bar i} + {\partial_2 \varphi}\parc{e^{\alpha t}\bar s,
e^{\alpha t}\bar n,e^{\alpha t}\bar i}} - \bar\phi_1\parc{\mu
+\bar{\mathbbm{V}}}+\bar\phi_4\bar{\mathbbm{V}}\\
&= \beta\left[\phi_1\parc{{\partial_1 \varphi}\parc{e^{\alpha t}s,
e^{\alpha t}n,e^{\alpha t}i} + {\partial_2 \varphi}\parc{e^{\alpha t}s,
e^{\alpha t}n,e^{\alpha t}i}}\right.\\
&\left. \quad -\bar\phi_1 \parc{{\partial_1 \varphi}\parc{e^{\alpha t}\bar s,
e^{\alpha t}\bar n,e^{\alpha t}\bar i} + {\partial_2 \varphi}\parc{e^{\alpha t}\bar s,
e^{\alpha t}\bar n,e^{\alpha t}\bar i}}\right]\\
&\quad -\beta\left[\phi_2 \parc{{\partial_1 \varphi}\parc{e^{\alpha t}s,
e^{\alpha t}n,e^{\alpha t}i} + {\partial_2 \varphi}\parc{e^{\alpha t}s,
e^{\alpha t}n,e^{\alpha t}i}}\right.\\
&\left. \quad -\bar\phi_2 \parc{{\partial_1 \varphi}\parc{e^{\alpha t}\bar s,
e^{\alpha t}\bar n,e^{\alpha t}\bar i} + {\partial_2 \varphi}\parc{e^{\alpha t}\bar s,
e^{\alpha t}\bar n,e^{\alpha t}\bar i}}\right]\\
&\quad + \mu(\phi_1- \bar\phi_1)+ \phi_1\mathbbm{V}
- \bar\phi_1\bar{\mathbbm{V}}-\phi_4\mathbbm{V}+\bar\phi_4\bar{\mathbbm{V}}.
\end{split}
\]
Multiplying by $\phi_1-\bar{\phi}_1$ and integrating from $0$ to $T$, we obtain
\begin{equation}
\label{eq:another-estimate1}
\begin{split}
&-\frac{1}{2}(\phi_1(0)-\bar {\phi}_1(0))^2 - \alpha \int_0^T (\phi_1-\bar{\phi}_1)^2 dt\\
& = \int_0^T \beta (\phi_1-\bar{\phi}_1) \left[\phi_1 \parc{{\partial_1 \varphi}\parc{e^{\alpha t}s,
e^{\alpha t}n,e^{\alpha t}i} + {\partial_2 \varphi}\parc{e^{\alpha t}s,
e^{\alpha t}n,e^{\alpha t}i}} \right.\\
& \left. \quad - \bar{\phi}_1 \parc{{\partial_1 \varphi}\parc{e^{\alpha t}\bar{s},
e^{\alpha t}\bar{n},e^{\alpha t}\bar{i}} + {\partial_2 \varphi}\parc{e^{\alpha t}\bar{s},
e^{\alpha t}\bar{n},e^{\alpha t}\bar{i}}}\right]dt\\
& \quad - \int_0^T \beta (\phi_1-\bar{\phi}_1) \left[\phi_2
\parc{{\partial_1 \varphi}\parc{e^{\alpha t}s,e^{\alpha t}n,e^{\alpha t}i}
+ {\partial_2 \varphi}\parc{e^{\alpha t}s,e^{\alpha t}n,e^{\alpha t}i}}\right.\\
&\left. \quad -\bar{\phi}_2 \parc{{\partial_1 \varphi}\parc{e^{\alpha t}\bar{s},
e^{\alpha t}\bar{n},e^{\alpha t}\bar{i}} + {\partial_2 \varphi}\parc{e^{\alpha t}\bar{s},
e^{\alpha t}\bar{n},e^{\alpha t}\bar{i}}}\right]dt\\
& \quad + \int_0^T \mu (\phi_1-\bar{\phi}_1)^2dt
+ \int_0^T (\phi_1-\bar{\phi}_1)(\phi_1\mathbbm{V}
- \bar\phi_1\bar{\mathbbm{V}}) dt\\
& \quad - \int_0^T (\phi_1-\bar{\phi}_1)(\phi_4\mathbbm{V}
-\bar\phi_4\bar{\mathbbm{V}}) dt.
\end{split}
\end{equation}
Multiplying~\eqref{eq:another-estimate1} by $-1$, we obtain that
\[
\begin{split}
&\frac{1}{2}(\phi_1(0)-\bar {\phi}_1(0))^2 + \alpha \int_0^T (\phi_1-\bar{\phi}_1)^2 dt\\
& \le \beta^u\int_0^T |\phi_1-\bar{\phi}_1| \left[|\phi_1 \partial_1 \varphi\parc{e^{\alpha t}s,
e^{\alpha t}n,e^{\alpha t}i} - \bar{\phi}_1 \partial_1 \varphi\parc{e^{\alpha t}\bar{s},
e^{\alpha t}\bar{n},e^{\alpha t}\bar{i}}|\right.\\
&\left. \quad + |\phi_1\partial_2 \varphi\parc{e^{\alpha t}s,e^{\alpha t}n,
e^{\alpha t}i} - \bar{\phi}_1 \partial_2 \varphi\parc{e^{\alpha t}\bar{s},
e^{\alpha t}\bar{n},e^{\alpha t}\bar{i}}|\right]dt\\
& \quad + \beta^u \int_0^T |\phi_1-\bar{\phi}_1| \left[|\phi_2
\partial_1 \varphi\parc{e^{\alpha t}s,e^{\alpha t}n,e^{\alpha t}i}
- \bar{\phi}_2 \partial_1 \varphi\parc{e^{\alpha t}\bar{s},
e^{\alpha t}\bar{n},e^{\alpha t}\bar{i}}|\right.\\
&\left. \quad + |\phi_2 \partial_2 \varphi\parc{e^{\alpha t}s,
e^{\alpha t}n,e^{\alpha t}i} - \bar{\phi}_2 {\partial_2
\varphi}\parc{e^{\alpha t}\bar{s},e^{\alpha t}\bar{n},
e^{\alpha t}\bar{i}}|\right]dt\\
& \quad - \int_0^T (\phi_1-\bar{\phi}_1)(\phi_1\mathbbm{V}
- \bar\phi_1\bar{\mathbbm{V}}) dt+ \int_0^T
(\phi_1-\bar{\phi}_1)(\phi_4\mathbbm{V}-\bar\phi_4\bar{\mathbbm{V}}) dt
\end{split}
\]
and thus, by~\eqref{eq:maj-partial-i}
and \eqref{eq:maj-partial-ij}, we conclude that
\[
\begin{split}
&\frac{1}{2}(\phi_1(0)-\bar {\phi}_1(0))^2
+ \alpha \int_0^T (\phi_1-\bar{\phi}_1)^2 dt\\
&\le \beta^u\int_0^T |\phi_1-\bar{\phi}_1| \left[\phi_1^u (M_{11}^u
e^{\alpha t}|s-\bar{s}|+M_{12}^u e^{\alpha t}|n-\bar{n}|
+M_{13}^u e^{\alpha t}|i-\bar{i}|) \right.\\
& \quad + M_1^u|\phi_1-\bar{\phi}_1|+ \phi_1^u (M_{21}^u
e^{\alpha t}|s-\bar{s}|+M_{22}^u e^{\alpha t}|n-\bar{n}|
+M_{23}^u e^{\alpha t}|i-\bar{i}|)\\
&\left. \quad + M_2^u|\phi_1-\bar{\phi}_1|\right]dt
+ \beta^u \int_0^T |\phi_1-\bar{\phi}_1| \left[\phi_2^u
\left(M_{11}^u e^{\alpha t}|s-\bar{s}|+M_{12}^u e^{\alpha t}|n-\bar{n}|\right)\right.\\
& \quad +M_{13}^u e^{\alpha t}|i-\bar{i}|)+ M_1^u|\phi_2
-\bar{\phi}_2|+\phi_2^u \left(M_{21}^u e^{\alpha t}|s
-\bar{s}|+M_{22}^u e^{\alpha t}|n-\bar{n}|\right.\\
&\left. \quad +M_{23}^u e^{\alpha t}|i-\bar{i}|)
+ M_2^u|\phi_2-\bar{\phi}_2|\right]dt
- \int_0^T (\phi_1-\bar{\phi}_1)(\phi_1\mathbbm{V}
- \bar\phi_1\bar{\mathbbm{V}}) dt\\
& \quad+ \int_0^T (\phi_1-\bar{\phi}_1)(\phi_4\mathbbm{V}
-\bar\phi_4\bar{\mathbbm{V}}) dt.
\end{split}
\]
Finally, we have
\[
\begin{split}
&\frac{1}{2}(\phi_1(0)-\bar {\phi}_1(0))^2
+ \alpha \int_0^T (\phi_1-\bar{\phi}_1)^2 dt\\
&\le \beta^u \phi_1^u e^{\alpha T} \left((M_{11}^u+M_{21}^u)
\int_0^T (\phi_1-\bar{\phi}_1)^2+ (s-\bar{s})^2 dt +(M_{12}^u+M_{22}^u) \right.\\
&\quad \left.\times \int_0^T (\phi_1-\bar{\phi}_1)^2 + (n-\bar{n})^2 dt
+ (M_{13}+M_{23})^u \int_0^T (\phi_1-\bar{\phi}_1)^2 + (i-\bar{i})^2 dt\right)\\
&\quad + \beta^u(M_1^u + M_2^u)\int_0^T (\phi_1-\bar{\phi}_1)^2dt
+ \beta^u \phi_2^u e^{\alpha T} \left((M_{11}^u+M_{21}^u)
\int_0^T (\phi_1-\bar{\phi}_1)^2 \right.\\
&\left.\quad + (s-\bar{s})^2 dt+(M_{12}^u+M_{22}^u)
\int_0^T (\phi_1-\bar{\phi}_1)^2 + (n-\bar{n})^2 dt
+ (M_{13}+M_{23})^u \right.\\
& \quad \left.\times \int_0^T (\phi_1-\bar{\phi}_1)^2
+ (i-\bar{i})^2 dt\right) + \beta^u(M_1^u + M_2^u)
\int_0^T (\phi_1-\bar{\phi}_1)^2+(\phi_2-\bar{\phi}_2)^2dt\\
& \quad + \int_0^T K_5[(\mathbbm{V}-\bar{\mathbbm{V}}|)^2
+2(\phi_1-\bar{\phi}_1)^2] dt\\&\quad \quad + \int_0^T K_6[(\mathbbm{V}
-\bar{\mathbbm{V}})^2+(\phi_4-\bar\phi_4)^2
+(\phi_1-\bar{\phi}_1)^2] dt
\end{split}
\]
\[
\begin{split}
& \le C_5 \e^{\alpha T} \int_0^T (s-\bar{s})^2+(e-\bar{e})^2+(i-\bar{i})^2
+(r-\bar{r})^2+(\phi_1-\bar{\phi}_1)^2+(\phi_2-\bar{\phi}_2)^2\\
&\quad +(\phi_4-\bar{\phi}_4)^2 dt + (K_5+K_6)
\int_0^T  (\mathbbm{V}-\bar {\mathbbm{V}})^2 dt,
\end{split}
\]
where $K_5$ and $K_6$ depend on the bounds for
$\bar p_1$, $\bar p_4$ and $\mathbbm{V}$ and
\[
\begin{split}
& C_5=\beta^u \phi_1^u \left((M_{11}^u+M_{21}^u) +2(M_{12}^u+M_{22}^u)+(M_{13}+M_{23})^u\right)\\
&\quad \quad +\beta^u \phi_2^u \left((M_{11}^u+M_{21}^u)+2(M_{12}^u+M_{22}^u)+(M_{13}+M_{23})^u\right)\\
&\quad \quad + 2\beta^u(M_1^u + M_2^u)+2K_5+K_6.
\end{split}
\]
By~\eqref{eq:bound-V-Vbar} and~\eqref{eq:bound-T-Tbar}, we obtain
\begin{equation}
\label{eq:phi1-bar-phi1}
\begin{split}
&\frac{1}{2}(\phi_1(0)-\bar {\phi}_1(0))^2 + \alpha \int_0^T (\phi_1-\bar{\phi}_1)^2 dt\\
& \le C_5 \e^{\alpha T} \int_0^T (s-\bar{s})^2
+(e-\bar{e})^2+(i-\bar{i})^2+(r-\bar{r})^2+(\phi_1-\bar{\phi}_1)^2\\
& \quad +(\phi_2-\bar{\phi}_2)^2+(\phi_4-\bar{\phi}_4)^2 dt\\
&\quad + (K_5+K_6)C_9\e^{2\alpha T}\int_0^T (s-\bar s)^2
+(\phi_1-\bar \phi_1)^2+(\phi_4-\bar \phi_4)^2 dt\\
& \le (C_5\e^{\alpha T}+(K_5+K_6)C_9\e^{2\alpha T})
\int_0^T (s-\bar{s})^2+(e-\bar{e})^2+(i-\bar{i})^2+(r-\bar{r})^2\\
&\quad +(\phi_1-\bar{\phi}_1)^2+(\phi_2-\bar{\phi}_2)^2+(\phi_4-\bar{\phi}_4)^2 dt\\
& \le (C_5+(K_5+K_6)C_9)\e^{2\alpha T} \int_0^T \Phi(t)+\Psi(t) dt.
\end{split}
\end{equation}
On the other hand, from the second equation
of \eqref{eq:Pontryagin-SEIR-Mayer-1}, one has
$-\alpha e^{-\alpha t}\phi_2+e^{-\alpha t} \dot{\phi}_2
= e^{-\alpha t}(\phi_1-\phi_2)\beta {\partial_2 \varphi}
\parc{e^{\alpha t}s,e^{\alpha t}n,e^{\alpha t}i}
+e^{-\alpha t}\phi_2\parc{\mu +\eps} - e^{-\alpha t}\phi_3\eps$.
Straightforward computations show that
\begin{equation}
\label{eq:phi2-bar-phi2}
\begin{split}
\frac{1}{2}(\phi_2(0)-\bar {\phi}_2(0))^2
+ \alpha \int_0^T (\phi_2-\bar{\phi}_2)^2 dt
& \le \left(C_6 e^{\alpha t}+ \eps^u\right) \int_0^T \Phi(t)+\Psi(t) dt,
\end{split}
\end{equation}
where 
\[
\begin{split}
& C_6=\beta^u \phi_1^u \left(M_{21}^u+2M_{22}+M_{23}^u\right)
+\beta^u \phi_2^u \left(M_{21}^u+2M_{22}^u+M_{23}^u\right)
+2\beta^uM_2^u.
\end{split}
\]
From the third equation of \eqref{eq:Pontryagin-SEIR-Mayer-1}, we conclude that
\[
\begin{split}
-\alpha e^{-\alpha t}\phi_3+e^{-\alpha t} \dot{\phi}_3
& = e^{-\alpha t}\phi_3\parc{\mu+\gamma+\mathbbm{T}}
+e^{-\alpha t}(\phi_1-\phi_2)\beta ({\partial_2 \varphi}
\parc{e^{\alpha t}s,e^{\alpha t}n,e^{\alpha t}i}\\
&\quad + {\partial_3 \varphi}\parc{e^{\alpha t}s,
e^{\alpha t}n,e^{\alpha t}i})
-e^{-\alpha t}\phi_4 (\gamma+\mathbbm{T})-\kappa_1
\end{split}
\]
and, by similar computations as the ones done for
the first equation of \eqref{eq:Pontryagin-SEIR-Mayer-1},
\begin{equation}
\label{eq:phi3-bar-phi3}
\begin{split}
\frac{1}{2}&(\phi_3(0)-\bar {\phi}_3(0))^2
+ \alpha \int_0^T (\phi_3-\bar{\phi}_3)^2 dt\\
& \le (C_7+(K_7+K_8)C_{10})\e^{2 \alpha t} \int_0^T \Phi(t)+\Psi(t) dt,
\end{split}
\end{equation}
where $K_7$ and $K_8$ depend on the bounds
for $\bar p_3$, $\bar p_4$ and $\mathbbm{T}$ and
\[
\begin{split}
& C_7=\beta^u \phi_1^u \left((M_{21}^u+M_{31}^u)
+2(M_{22}^u+M_{32}^u)+(M_{23}+M_{33})^u\right)\\
&\quad \quad +\beta^u \phi_2^u \left((M_{21}^u+M_{31}^u)
+2(M_{22}^u+M_{32}^u)+(M_{23}+M_{33})^u\right)\\
&\quad \quad +2\beta^u(M_2^u + M_3^u)+\gamma^u+2K_7+K_8.
\end{split}
\]
From the fourth equation of \eqref{eq:Pontryagin-SEIR-Mayer-1},
\begin{equation*}
\begin{split}
-\alpha &e^{-\alpha t}\phi_4+e^{-\alpha t} \dot{\phi}_4\\
&= e^{-\alpha t}(\phi_1-\phi_2)\beta {\partial_2 \varphi}
\parc{e^{\alpha t}s,e^{\alpha t}n,e^{\alpha t}i}
+e^{-\alpha t}\phi_4\parc{\mu +\eta} - e^{-\alpha t}\phi_1\eta
\end{split}
\end{equation*}
and, by similar computations as the ones done for
the second equation of \eqref{eq:Pontryagin-SEIR-Mayer-1},
\begin{equation}
\label{eq:phi4-bar-phi4}
\begin{split}
\frac{1}{2}(\phi_4(0)-\bar {\phi}_4(0))^2
+ \alpha \int_0^T (\phi_4-\bar{\phi}_4)^2 dt
& \le \left(C_8 e^{\alpha t}+\eta^u\right) \int_0^T \Phi(t)+\Psi(t) dt,
\end{split}
\end{equation}
where
\[
\begin{split}
& C_8=\beta^u \phi_1^u \left(M_{21}^u+2M_{22}+M_{23}^u\right)
+\beta^u \phi_2^u \left(M_{21}^u+2M_{22}^u+M_{23}^u\right)
+2\beta^uM_2^u.
\end{split}
\]
We now obtain the bounds
for $\parc{\mathbbm{V}-\bar{\mathbbm{V}}}^2$
and $\parc{\mathbbm{T}-\bar{\mathbbm{T}}}^2$ announced
in~\eqref{eq:bound-V-Vbar} and~\eqref{eq:bound-T-Tbar}. We have
\[
\begin{split}
\parc{\mathbbm{V}-\bar{\mathbbm{V}}}^2
& = \left(\frac{e^{\alpha t}s}{2\kappa_3}(e^{-\alpha t}\phi_1
-e^{-\alpha t}\phi_4) -\frac{e^{\alpha t}\bar s}{2\kappa_3}
(e^{-\alpha t}\bar\phi_1-e^{-\alpha t}\bar\phi_4)\right)^2\\
& = \frac{1}{4\kappa_3^2}\left(s\phi_1-s\phi_4-\bar s\bar\phi_1
+\bar s\bar\phi_4\right)^2\\
& = \frac{1}{4\kappa_3^2}\left(s(\phi_1-\bar\phi_1)+(s-\bar s)\bar\phi_1
+s(\bar\phi_4-\phi_4)+(-s+\bar s)\bar\phi_4\right)^2
\end{split}
\]
and thus
\begin{equation}
\label{majV}
\begin{split}
&\parc{\mathbbm{V}-\bar{\mathbbm{V}}}^2
= \frac{1}{4\kappa_3^2}((s(\phi_1-\bar\phi_1)+(s-\bar s)\bar\phi_1)^2
+2(s(\phi_1-\bar\phi_1)+(s-\bar s)\bar\phi_1)\\
& \quad \times(s(\bar\phi_4-\phi_4)+(\bar s-s)\bar\phi_4)
+(s(\bar\phi_4-\phi_4)+(\bar s-s)\bar\phi_4)^2)\\
& = \Bigl[s^2(\phi_1-\bar\phi_1)^2
+2s\bar\phi_1(s-\bar s)(\phi_1-\bar\phi_1)+\bar\phi_1^2(s-\bar s)^2
+2s^2(\phi_1-\bar\phi_1)(\bar\phi_4-\phi_4)\\
&\quad +2s\bar\phi_1(s-\bar s)(\bar\phi_4-\phi_4)
+2\bar\phi_4s(s-\bar s)(\bar\phi_1-\phi_1)
-2\bar\phi_1\bar\phi_4(s-\bar s)^2\\
&\quad +s^2(\bar\phi_4-\phi_4)^2+2\bar\phi_4
s(\bar\phi_4-\phi_4)(\bar s-s)+\bar\phi_4^2(\bar s- s)^2\Bigr] 
\frac{1}{4\kappa_3^2}\\
& \le \frac{1}{4\kappa_3^2}((4s\bar\phi_1+\bar\phi_1^2
+4s\bar\phi_4+\bar\phi_4^2)(s-\bar s)^2
+(3s^2+2s\bar\phi_1+2s\bar\phi_4)(\phi_1-\bar\phi_1)^2\\
&\quad +(3s^2+2s\bar\phi_1+2s\bar\phi_4)(\bar\phi_4-\phi_4)^2)\\
& \le C_9 \e^{2\alpha T} [(s-\bar s)^2
+(\phi_1-\bar{\phi}_1)^2+(\phi_4-\bar{\phi}_4)^2],\\
\end{split}
\end{equation}
where
$$
C_9 = \frac{1}{4\kappa_2^2}\left(4\max\{S\}\max\{\bar p_1+\bar p_4\}
+\max\{\bar p_1\}^2+\max\{\bar p_4\}^2+3\max\{S\}^2\right).
$$
Analogously, we obtain
\begin{equation}
\label{majT}
\parc{\mathbbm{T}-\bar{\mathbbm{T}}}^2 \le C_{10} \e^{2\alpha T}[(i-\bar i)^2
+(\phi_3-\bar{\phi}_3)^2+(\phi_4-\bar{\phi}_4)^2],
\end{equation}
where
$$
C_{10} = \frac{1}{4\kappa_2^2}(4\max\{I\}\max\{\bar p_3+\bar p_4\}
+ \max\{\bar p_3\}^2+\max\{\bar p_4\}^2+3\max\{I\}^2).
$$
Finally, we have all the bounds needed to prove our result.
Adding equations~\eqref{eq:s-bar-s}, \eqref{eq:e-bar-e},
\eqref{eq:i-bar-i}, \eqref{eq:r-bar-r}, \eqref{eq:phi1-bar-phi1},
\eqref{eq:phi2-bar-phi2}, \eqref{eq:phi3-bar-phi3}
and \eqref{eq:phi4-bar-phi4}, we obtain, for the sum of the left-hand sides,
$$
\frac{1}{2}\Psi(T)+\frac{1}{2}\Phi(0)+\alpha\int_0^T
\left(\Psi(T)+\Phi(T)\right) dt
$$
and thus
\begin{multline*}
\frac{1}{2}[\Psi(T)+\Phi(0)]+\alpha\int_0^T \left(\Psi(T)+\Phi(T)\right)dt\\
\le \widetilde{C}\int_0^T
\left(\Psi(T)+\Phi(T)\right)dt
+\widehat{C}e^{3\alpha T}\int_0^T\left(\Psi(T)+\Phi(T)\right)dt,
\end{multline*}
which is equivalent to
\begin{equation}
\label{ineq}
\begin{split}
&\frac{1}{2}[\Psi(T)+\Phi(0)]+
\left(\alpha-\widetilde{C}-\widehat{C}e^{3 \alpha T}\right)
\int_0^T\left(\Psi(T)+\Phi(T)\right)dt \le 0.
\end{split}
\end{equation}
We now choose $\alpha$ so that
$\alpha>\widetilde{C}+\widehat{C}$
and note that $\frac{\alpha-\widetilde{C}}{\widehat{C}}>1$.
Subsequently, we choose $T$ such that
$$
T<\frac{1}{3\alpha}\ln\left(\frac{\alpha-\widetilde{C}}{\widehat{C}}\right).
$$
Then,
\begin{equation}
\label{eq:final-estimate}
3 \alpha T<\ln\left(\frac{\alpha-\widetilde{C}}{\widehat{C}}\right)
\quad \Rightarrow \quad e^{3 \alpha T}<\frac{\alpha-\widetilde{C}}{\widehat{C}}.
\end{equation}
It follows that $\alpha-\widetilde{C}-\widehat{C}e^{3\alpha T}>0$,
so that inequality \eqref{ineq} can hold if and only if
we have $s(t)=\bar s(t)$, $e(t)=\bar e(t)$, $i(t)=\bar i(t)$,
$r(t)=\bar r(t)$, $\phi_1(t)=\bar \phi_1(t)$,
$\phi_2(t)=\bar \phi_2(t)$, $\phi_3(t)=\bar \phi_3(t)$
and $\phi_4(t)=\bar \phi_4(t)$ for all $t\in[0,T]$.
This is equivalent to $S(t)=\bar S(t)$, $E(t)=\bar E(t)$,
$I(t)=\bar I(t)$, $R(t)=\bar R(t)$, $p_1(t)=\bar p_1(t)$,
$p_2(t)=\bar p_2(t)$, $p_3(t)=\bar p_3(t)$ and $p_4(t)=\bar p_4(t)$.
With this, the uniqueness of the optimal control is established on the interval $[0,T]$.
If $T\ge t_f$, then we have uniqueness on the whole interval. Otherwise, we can obtain
uniqueness on $[T,2T]$ for the optimal control problem whose initial conditions
on time $T$ coincide with the values of $S$, $E$, $I$ and $R$ on the end-time of the
interval $[0,T]$ (note that, by the forward invariance of the set $\Gamma$,
and since the constants $\widetilde{C}$ and $\widehat{C}$ in \eqref{eq:final-estimate}
depend only on the values of the several state and co-state variables on $\Gamma$,
we still have the same $T$). Proceeding in the same way, we conclude, after a finite
number of steps, that we have uniqueness on the interval $[0,t_f]$. The proof is complete.
\end{proof}


\section{Numerical simulations}
\label{section:Simulation-control}

In what follows, the incidence into the exposed class of susceptible individuals
and the birth function $\Lambda\parc{t}$ are chosen, for illustrative purposes, to be
\begin{equation}
\label{eq:beta}
\beta(t)\varphi(S,N,I)=0.56(1 - \mathrm{per}\cdot \cos( 2 \pi  t + 0.26 ) )SI
\end{equation}
and
\begin{equation}
\label{eq:Lambda}
\Lambda\parc{t} = 0.05 + 0.05\, \mathrm{per} \cdot\cos(2\pi t)
\end{equation}
with $\mathrm{per}\in[0,1[$. The remaining parameter functions,
$\mu\parc{t}$, $\eta\parc{t}$, $\eps\parc{t}$ and $\gamma\parc{t}$,
are assumed constant. Their values, as well as several other
used in this section, were taken from \cite{Weber200195} and
\cite{Zhang-Liu-Teng-AM-2012} and are presented in Table~\ref{table1}.
\begin{table}[ht!]
\centering
\caption{\small{Values of the parameters for problem \eqref{seir003}
used in Section~\ref{section:Simulation-control}.}}
\begin{tabular}{lll}
\toprule
Name & Description & Value\\
\midrule
$S_0$ & Initial susceptible population & 0.98\\
$E_0$ & Initial exposed population & 0\\
$I_0$ & Initial infective population & 0.01\\
$R_0$ & Initial recovered population & 0.01\\
$\mu$ & natural deaths & 0.05\\
$\eps$ & infectivity rate & 0.03\\
$\gamma$ & rate of recovery & 0.05\\
$\eta$ & loss of immunity rate & 0.041\\
$k_1$ & weight for the number of infected & 1\\
$k_2$ & weight for treatment & 0.01\\
$k_3$ & weight for vaccination & 0.01\\
$\tau_{\max}$ & maximum rate of treatment & 0.1\\
$v_{\max}$ & maximum rate of vaccination & 0.4\\
\bottomrule
\end{tabular}
\label{table1}
\end{table}
Note that we are in the autonomous case when
$\mathrm{per} = 0$ and in a periodic situation
for $0 < \mathrm{per} < 1$.

The solutions to the optimal control problems \eqref{seir003} here considered
exist (Theorem~\ref{mr:thm:exist}), are unique (Theorem~\ref{thm:uniq})
and can be found using the optimality conditions given
in Theorem~\ref{thm:Pontryagin-SEIR-Mayer}. Precisely,
our method to solve the problem consists to use
the state equations (the four ordinary differential equations
of problem \eqref{seir003}), the initial conditions, the adjoint equations
\eqref{eq:Pontryagin-SEIR-Mayer-1} and the transversality conditions
\eqref{eq:Pontryagin-SEIR-Mayer-5} with the controls given by
\eqref{eq:Pontryagin-SEIR-Mayer-6} and ~\eqref{eq:Pontryagin-SEIR-Mayer-7},
which are substituted into the state and adjoint equations.
The state equations are solved with the initial conditions
of Table~\ref{table1}, while the adjoint system is solved
backward in time, with the change of variable
\begin{align}
\label{change_var}
t'=t_f-t.
\end{align}
Then, the two systems of equations constitute an initial value problem,
which is solved numerically with an explicit 4th and 5th order Runge--Kutta
method through the \texttt{ode45} solver of \textsf{MATLAB}.
The procedure is briefly described in Algorithm~\ref{alg}.
\begin{algorithm}
\label{alg}
1. Let $i=0$, $\mathbbm{T}_i=0$ and $\mathbbm{V}_i=0$.
	
2. Let $i=i+1$. The variables $S_i$, $E_i$, $I_i$
and $R_i$ are determined using the initial conditions and
the vectors $\mathbbm{T}_{i-1}$ and $\mathbbm{V}_{i-1}$.
	
3. Apply the change of variable \eqref{change_var}
to the adjoint system, to the state and control variables.
Compute the adjoint variables $p_{1,i}$, $p_{2,i}$, $p_{3,i}$ and $p_{4,i}$
solving the resulting adjoint system.
	
4. Compute the control variables $\mathbbm{T}_{i}$ and $\mathbbm{V}_{i}$
with formulas~\eqref{eq:Pontryagin-SEIR-Mayer-6}
and \eqref{eq:Pontryagin-SEIR-Mayer-7}.
	
5. If the relative error is smaller than a given tolerance
for all the variables ($<1\%$ in our case), then stop.
Otherwise, go to step 2.

\caption{Numerical algorithm to solve problem \eqref{seir003}.}
\end{algorithm}

In each plot of Figures~\ref{fig_SR_EI0} and \ref{fig_SR_EI08},
we present two sets of trajectories (distinguished by using dashed
and continuous lines), in order to easily compare the uncontrolled
and optimally controlled situations, the former mentioned by the
suffix ``\emph{-u}'' in the figures' legend.
\begin{figure}[ht]
\begin{minipage}[b]{.495\linewidth}
\includegraphics[width=\linewidth]{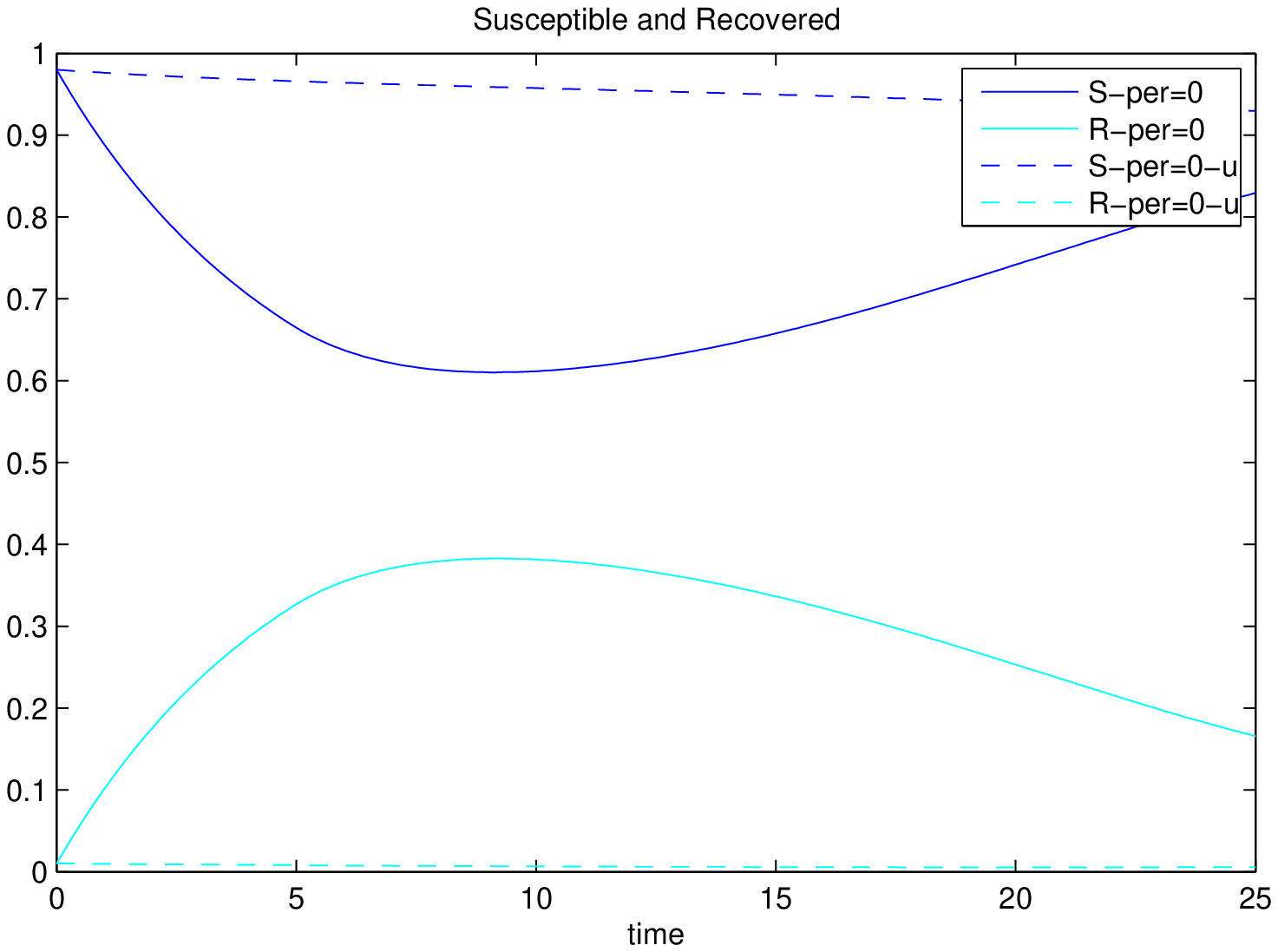}
\end{minipage} \hfill
\begin{minipage}[b]{.495\linewidth}
\includegraphics[width=\linewidth]{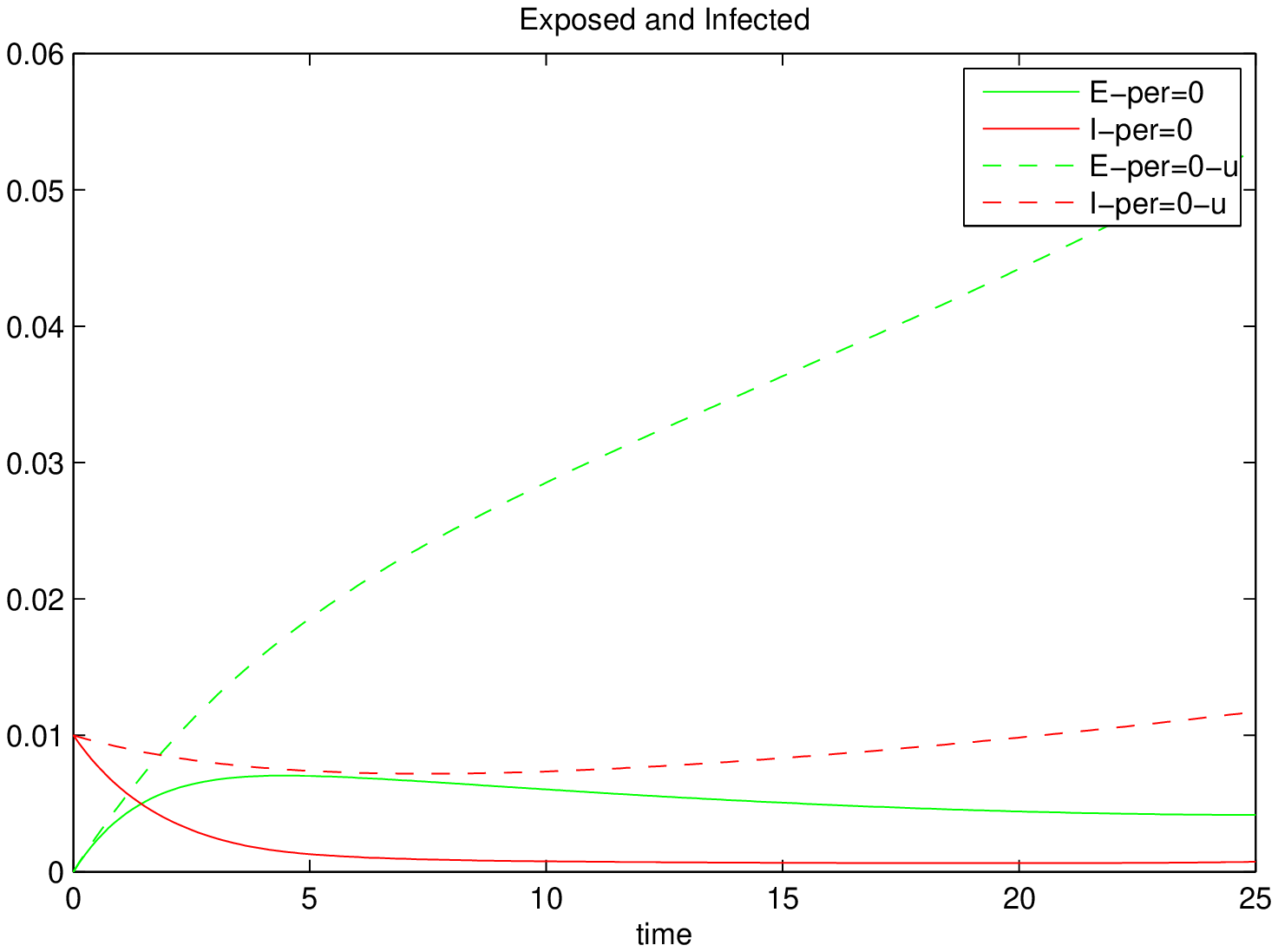}
\end{minipage}
\caption{\small{SEIRS autonomous model ($\mathrm{per}=0$ in
\eqref{eq:beta} and \eqref{eq:Lambda}):
uncontrolled (dashed lines) versus
optimally controlled (continuous lines).}}
\label{fig_SR_EI0}
\end{figure}
\begin{figure}[ht]
\begin{minipage}[b]{.495\linewidth}
\includegraphics[width=\linewidth]{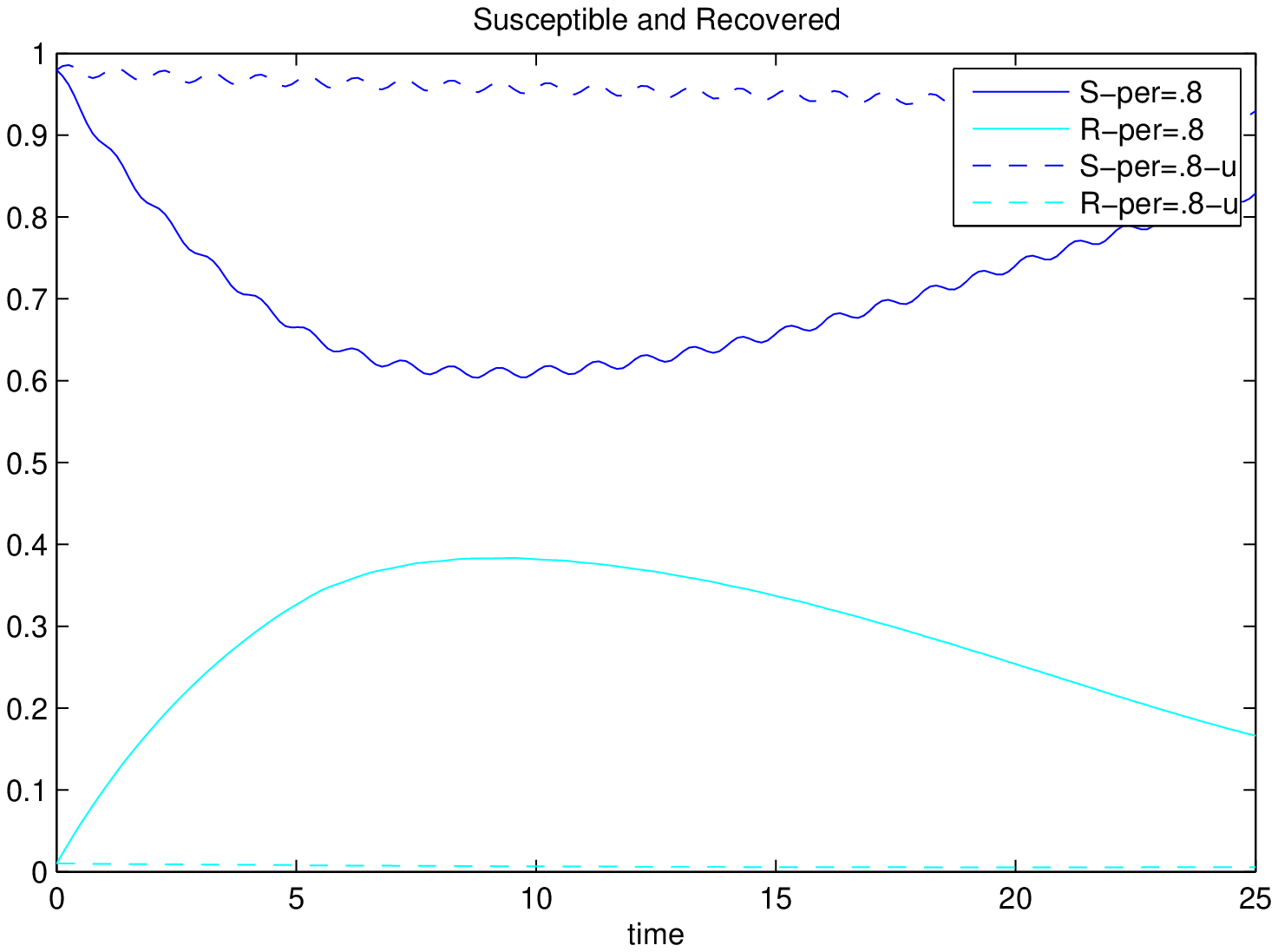}
\end{minipage} \hfill
\begin{minipage}[b]{.495\linewidth}
\includegraphics[width=\linewidth]{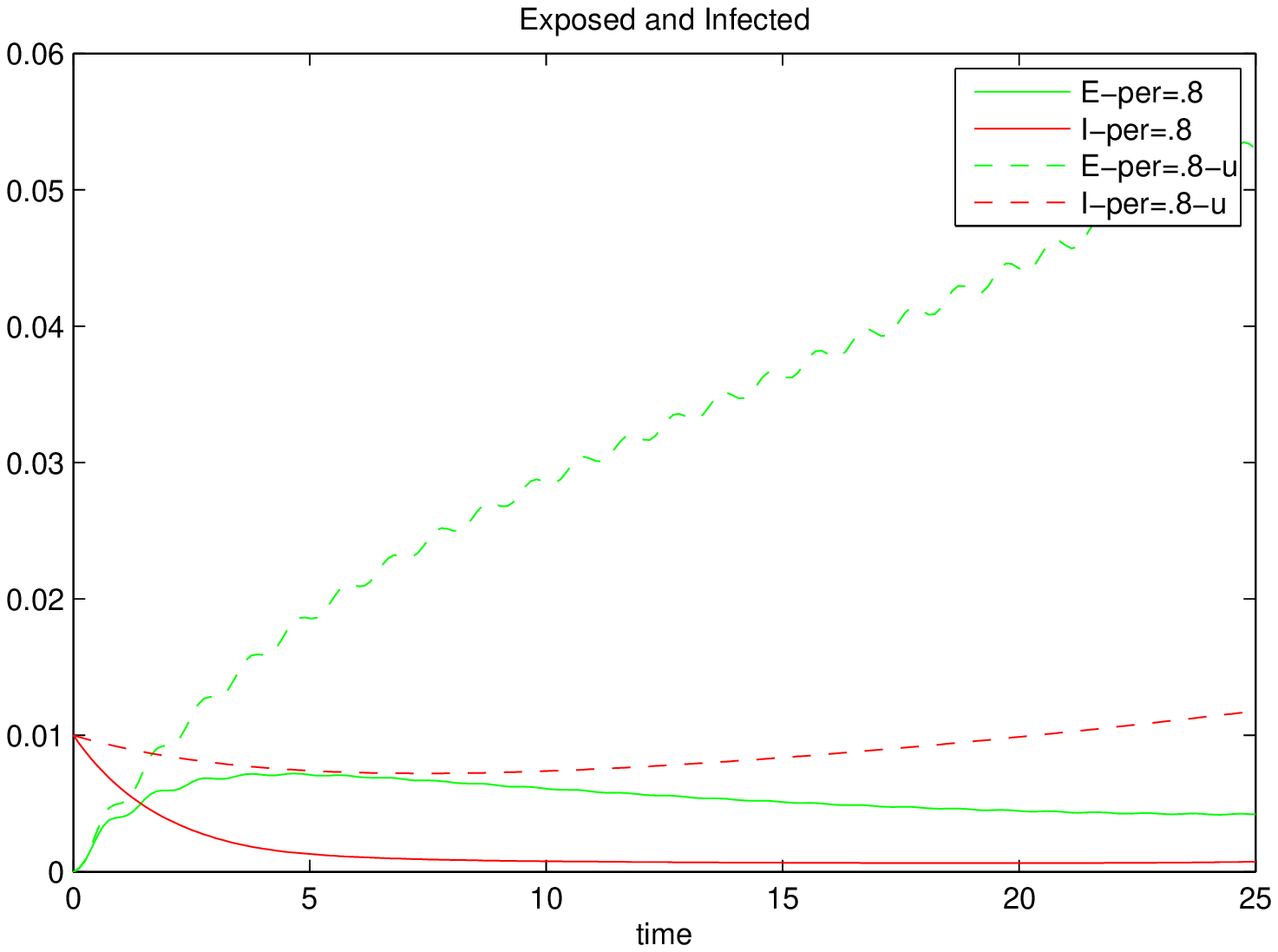}
\end{minipage}
\caption{\small{SEIRS periodic model ($\mathrm{per}=0.8$
in \eqref{eq:beta} and \eqref{eq:Lambda}):
uncontrolled (dashed lines) versus
optimally controlled (continuous lines).}}
\label{fig_SR_EI08}
\end{figure}
The behavior of our SEIRS model with both $\mathrm{per}=0$ (autonomous case,
Figure~\ref{fig_SR_EI0}) and $\mathrm{per}=0.8$ (periodic case, Figure~\ref{fig_SR_EI08}),
show the effectiveness of optimal control theory.
Indeed, in Figures~\ref{fig_SR_EI0} and \ref{fig_SR_EI08} we observe that,
if we apply treatment and vaccination as given by Theorem~\ref{thm:Pontryagin-SEIR-Mayer}
(optimally controlled case), then the number of exposed and infected individuals
is significantly lower, as well as the number of susceptible individuals,
while the number of recovered is significantly higher. It can be also seen
that the susceptible and recovered classes have a very different behavior
in the controlled and uncontrolled situations.

In Figures~\ref{fig_SR_EI_c} and \ref{fig_SR_EI_u}, we have the same
trajectories as in Figures~\ref{fig_SR_EI0} and \ref{fig_SR_EI08}.
They illustrate the effect of the periodicity of
$\Lambda(t)$ \eqref{eq:Lambda} and $\beta(t)$ \eqref{eq:beta}
in the classes $S$, $E$, $I$ and $R$ of individuals.
The effect is perceptible in susceptible and exposed classes,
since the periodic functions are present in these classes. From our results,
we claim that the periodicity effect is ``softened''
in the transition between classes.
\begin{figure}[ht]
\begin{minipage}[b]{.495\linewidth}
\includegraphics[width=\linewidth]{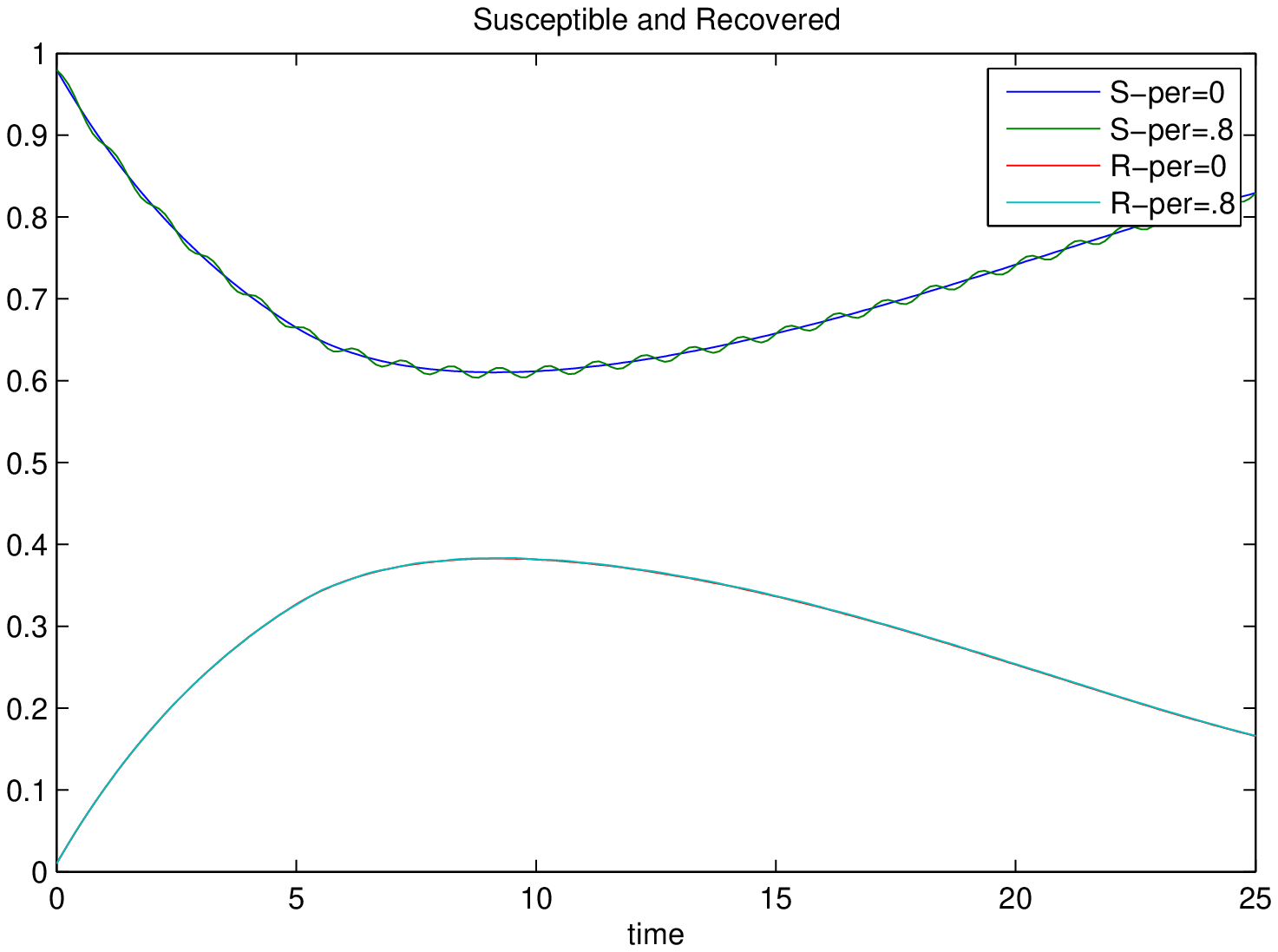}
\end{minipage} \hfill
\begin{minipage}[b]{.495\linewidth}
\includegraphics[width=\linewidth]{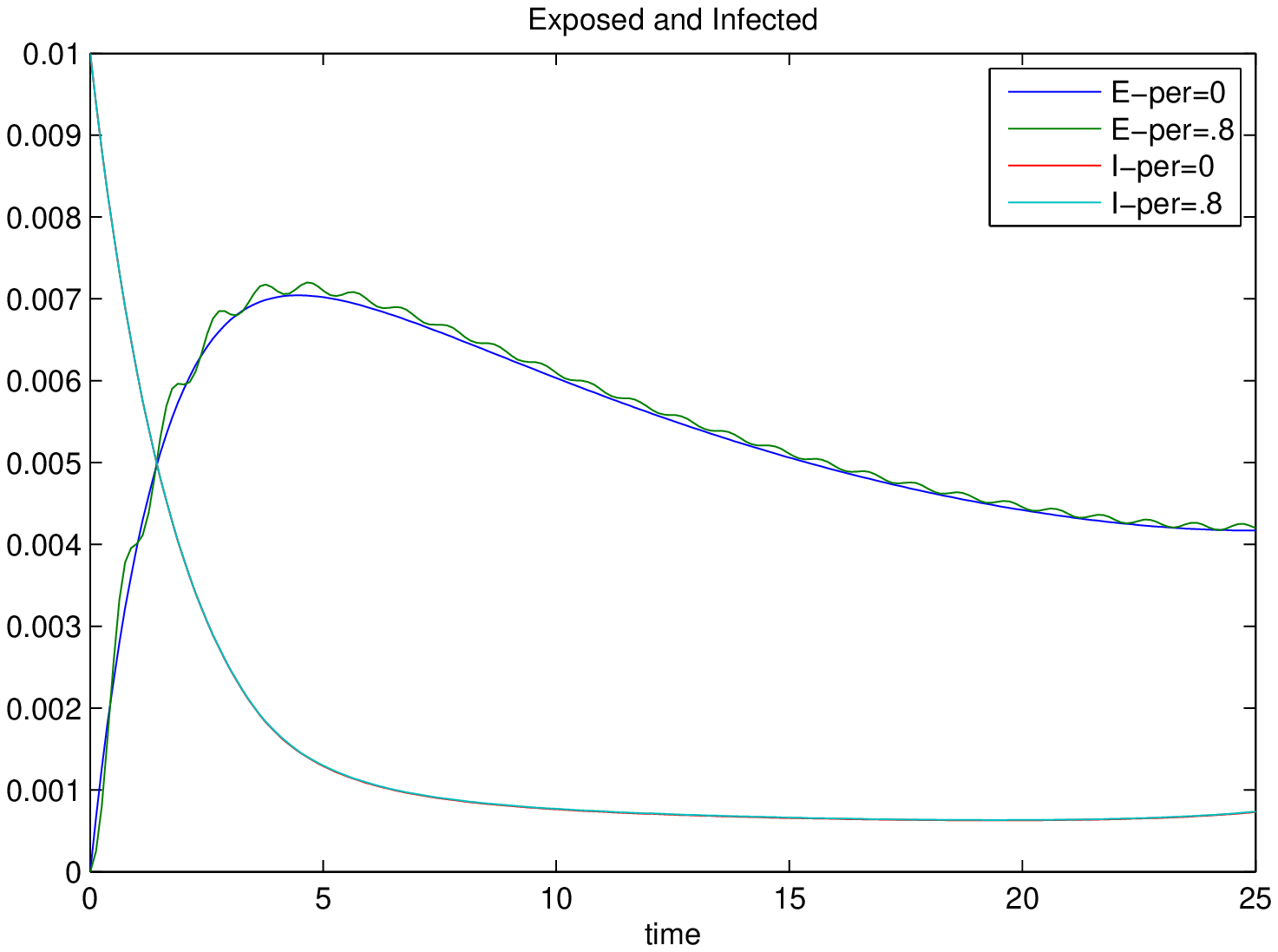}
\end{minipage}
\caption{\small{SEIRS model subject to optimal control:
autonomous ($\mathrm{per}=0$) versus periodic ($\mathrm{per}=0.8$) cases.}}
\label{fig_SR_EI_c}
\end{figure}
\begin{figure}[ht]
\begin{minipage}[b]{.495\linewidth}
\includegraphics[width=\linewidth]{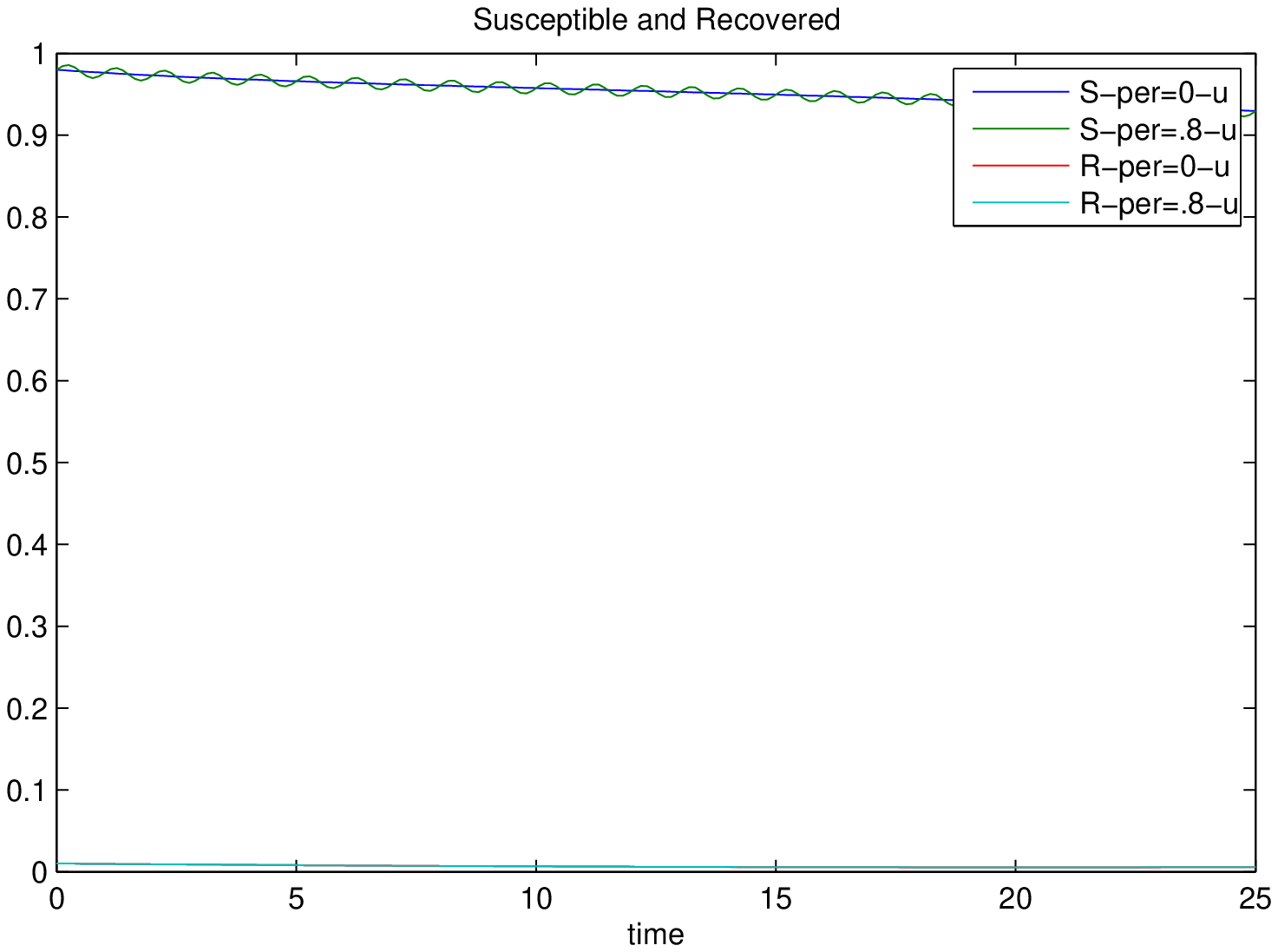}
\end{minipage} \hfill
\begin{minipage}[b]{.495\linewidth}
\includegraphics[width=\linewidth]{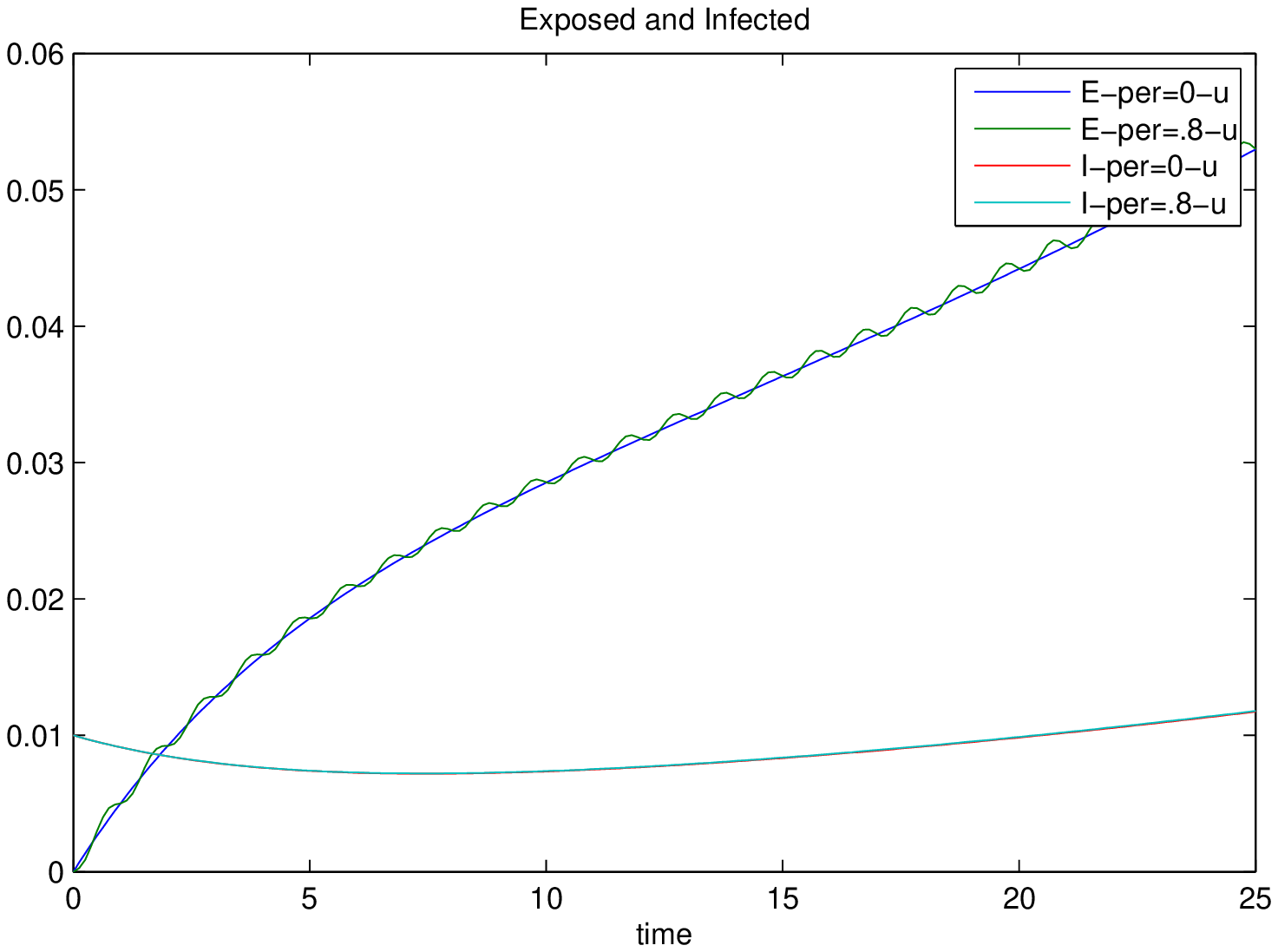}
\end{minipage}
\caption{\small{SEIRS model without control measures:
autonomous ($\mathrm{per}=0$) versus periodic ($\mathrm{per}=0.8$) cases.}}
\label{fig_SR_EI_u}
\end{figure}

In Figure~\ref{fig_t_v}, we show the optimal controls:
treatment $\mathbbm{T}^*(t)$ of infective individuals (left side)
and vaccination $\mathbbm{V}^*(t)$ of susceptible  (right side).
According to the minimality conditions \eqref{eq:Pontryagin-SEIR-Mayer-6}
and \eqref{eq:Pontryagin-SEIR-Mayer-7}, both controls go to zero when $t\to t_{f}=25$.
The periodicity effect is perceptible in $\mathbbm{V}^*(t)$,
consequence of the fact that vaccination takes place in the susceptible class.
Treatment occurs in the infective class and, as we have seen, in this class
the periodicity is not so perceptible. As a consequence, periodicity is only
slightly perceptible in the treatment control variable $\mathbbm{T}^*(t)$.
\begin{figure}[ht]
\begin{minipage}[b]{.495\linewidth}
\includegraphics[width=\linewidth]{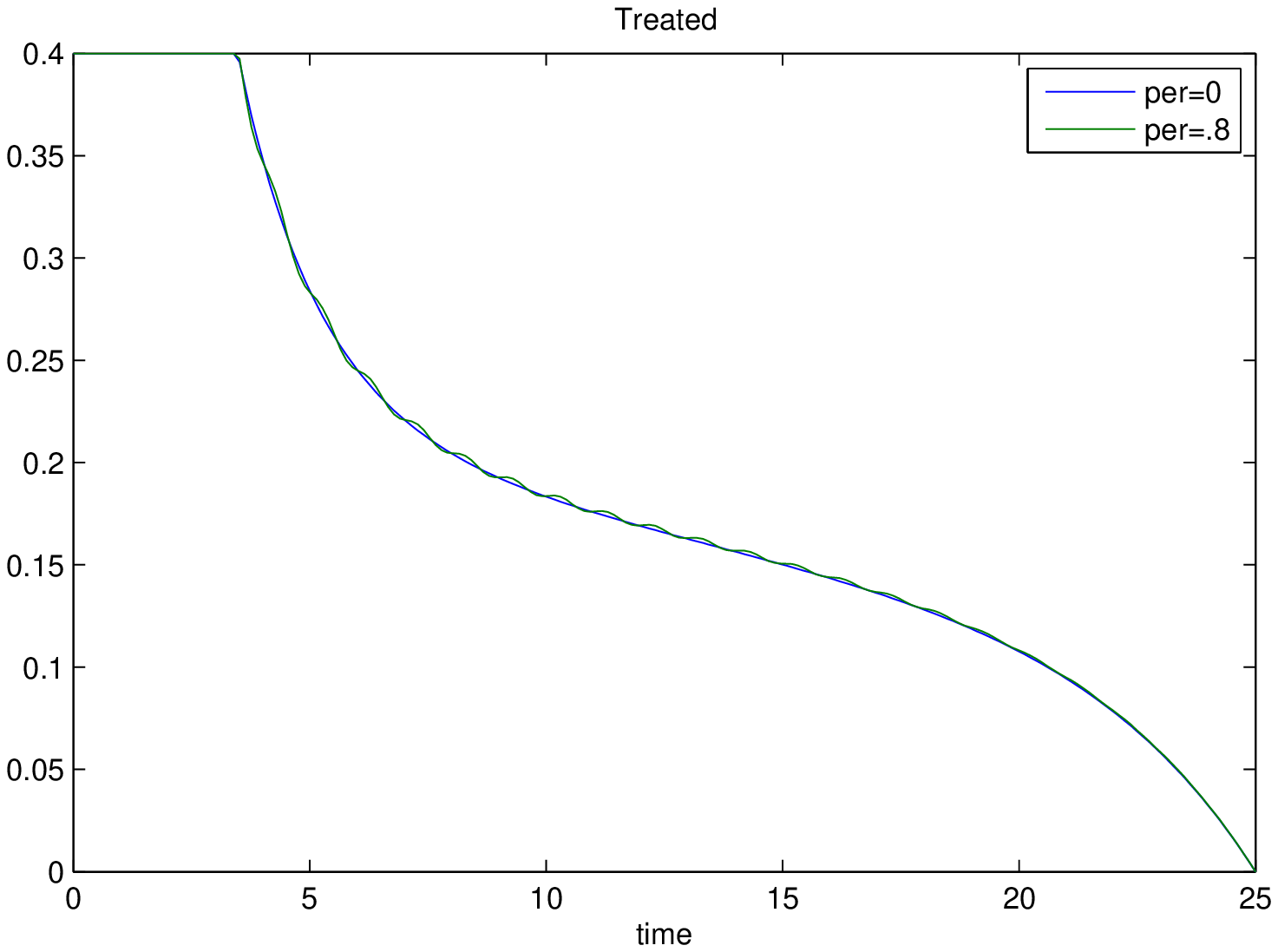}
\end{minipage} \hfill
\begin{minipage}[b]{.495\linewidth}
\includegraphics[width=\linewidth]{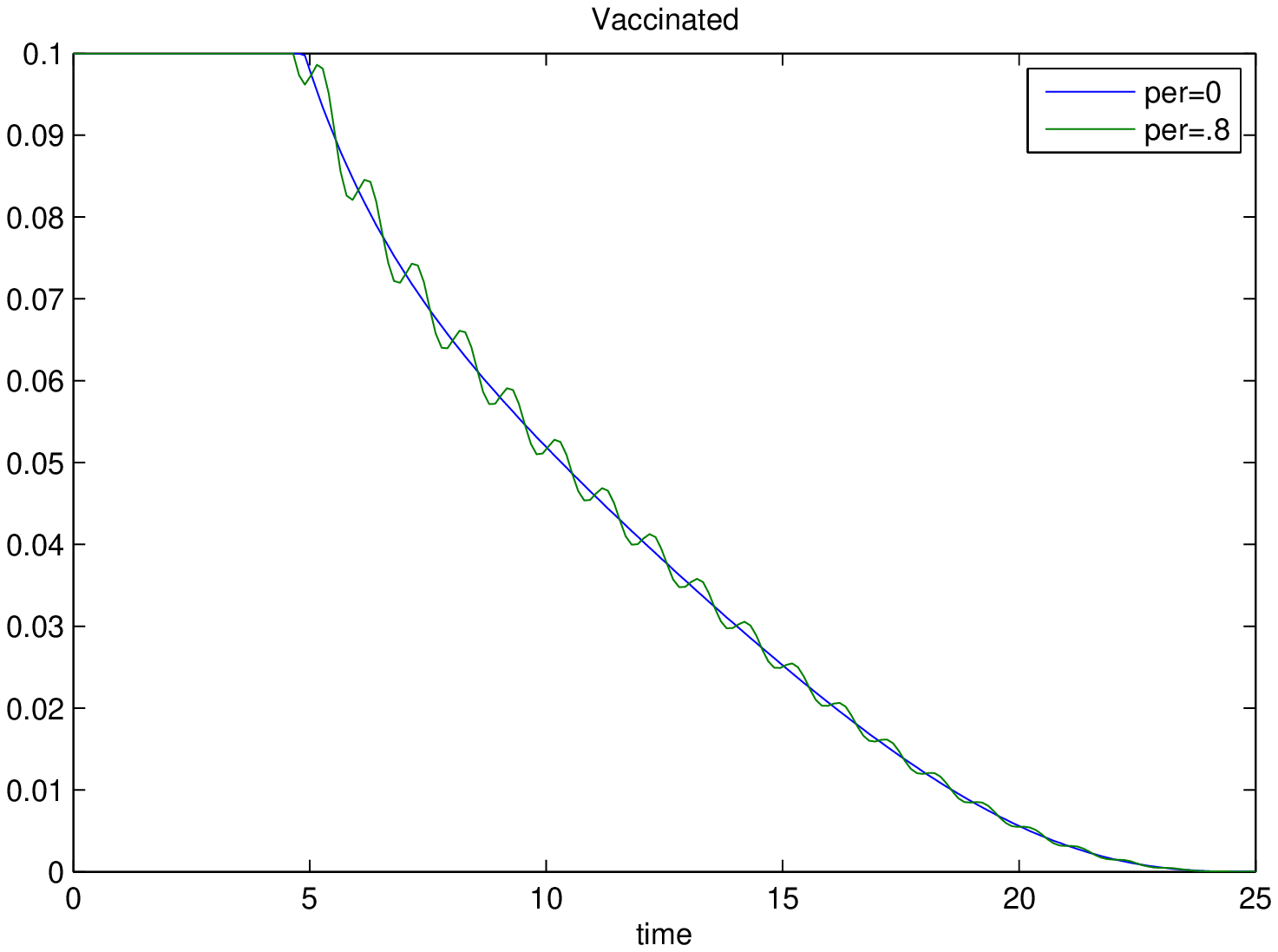}
\end{minipage}
\caption{\small{The optimal controls $\mathbbm{T}^*$
\eqref{eq:Pontryagin-SEIR-Mayer-6} (treatment)
and $\mathbbm{V}^*$ \eqref{eq:Pontryagin-SEIR-Mayer-7}
(vaccination): autonomous ($\mathrm{per}=0$)
versus periodic ($\mathrm{per}=0.8$) cases.}}
\label{fig_t_v}
\end{figure}

In Figures~\ref{fig_mu} to \ref{fig_eta}, we present the behavior of infected
individuals $I^*(t)$, treatment $\mathbbm{T}^*(t)$ and vaccination $\mathbbm{V}^*(t)$
optimal controls, when we vary the parameters $\mu$, $\gamma$, $\eps$ and $\eta$,
respectively, maintaining, in each case, the initial values and the remaining parameters
as in Table~\ref{table1}. In all Figures~\ref{fig_mu}--\ref{fig_eta},
we varied the respective parameter ($\mu$, $\gamma$, $\eps$ and $\eta$)
from $0$ to $0.1$ in steps of length $0.01$.

Referring to Figure~\ref{fig_mu}, where the variation of $\mu$ is analyzed,
we can see that the effect of periodicity is more perceptible in vaccination
than in treatment.  In the infected class, for low values of $\mu$ (low mortality),
we observe that the number of infected increases. This is justified
by the difference between birth and death.
\begin{figure}[ht]
\centering
\includegraphics[scale=0.43]{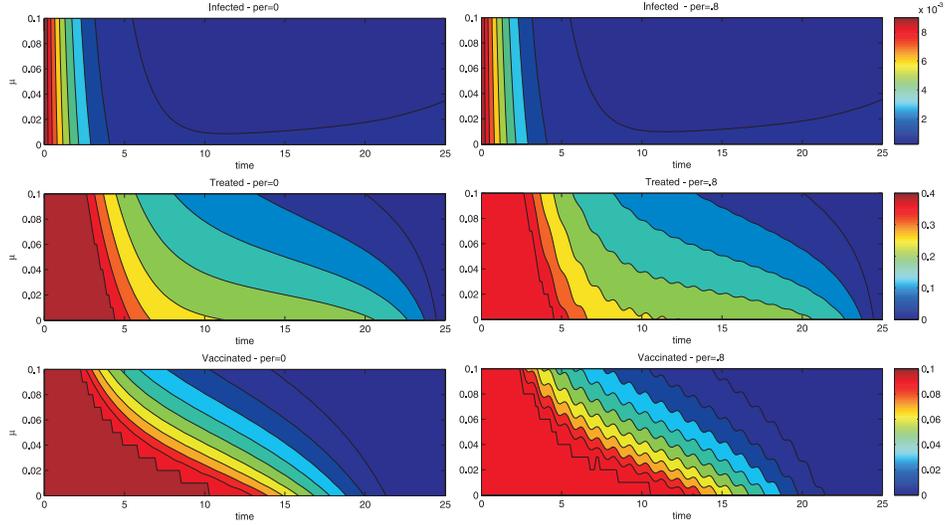}
\caption{\small{Variation of infected individuals $I^*(t)$
and optimal measures of treatment $\mathbbm{T}^*(t)$
and vaccination $\mathbbm{V}^*(t)$, in both autonomous ($\mathrm{per}=0$)
and periodic ($\mathrm{per}=0.8$) cases, with the natural death $\mu \in [0, 0.1]$.}}
\label{fig_mu}
\end{figure}

Concerning Figure~\ref{fig_gama}, where we vary
the rate of recovery $\gamma$,
the effect of periodicity is analogous
to the previous situation: the bigger the value of $\gamma$,
more infected individuals recover and thus the faster
the infected class decreases.
\begin{figure}[ht]
\centering
\includegraphics[scale=0.43]{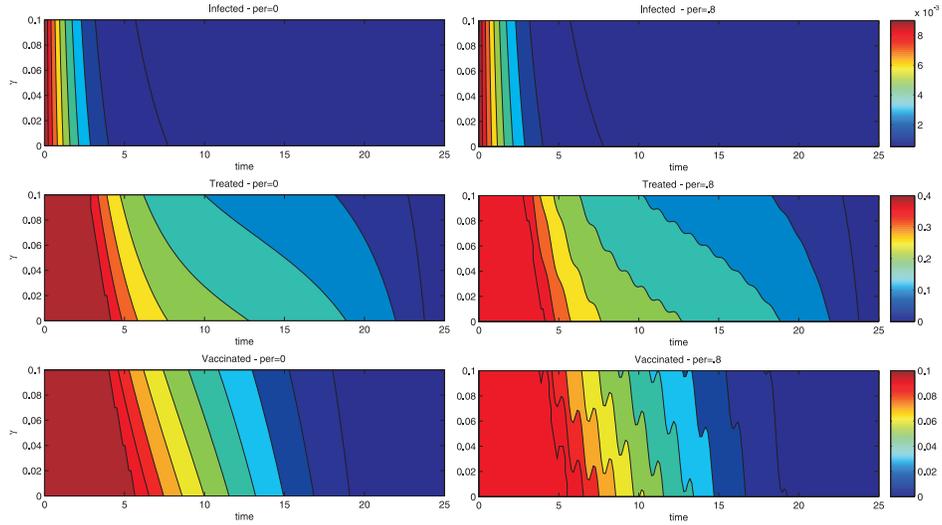}
\caption{\small{Variation of infected individuals $I^*(t)$
and optimal measures of treatment $\mathbbm{T}^*(t)$
and vaccination $\mathbbm{V}^*(t)$, in both autonomous ($\mathrm{per}=0$)
and periodic ($\mathrm{per}=0.8$) cases,  with the
rate of recovery $\gamma \in [0, 0.1]$.}}
\label{fig_gama}
\end{figure}

In Figure~\ref{fig_epsilon}, one can see the effect of the variation
of the infectivity rate $\eps$. The effect of periodicity is also
in this case analogous to the previous situations. Indeed, when
we have a higher value of $\eps$, we have a faster transition
of exposed individuals into the infected class and this is the reason
why we observe in Figure~\ref{fig_epsilon} that an increase
on the value of $\eps$ leads to an increase in the infected class.
\begin{figure}[ht]
\centering
\includegraphics[scale=0.55]{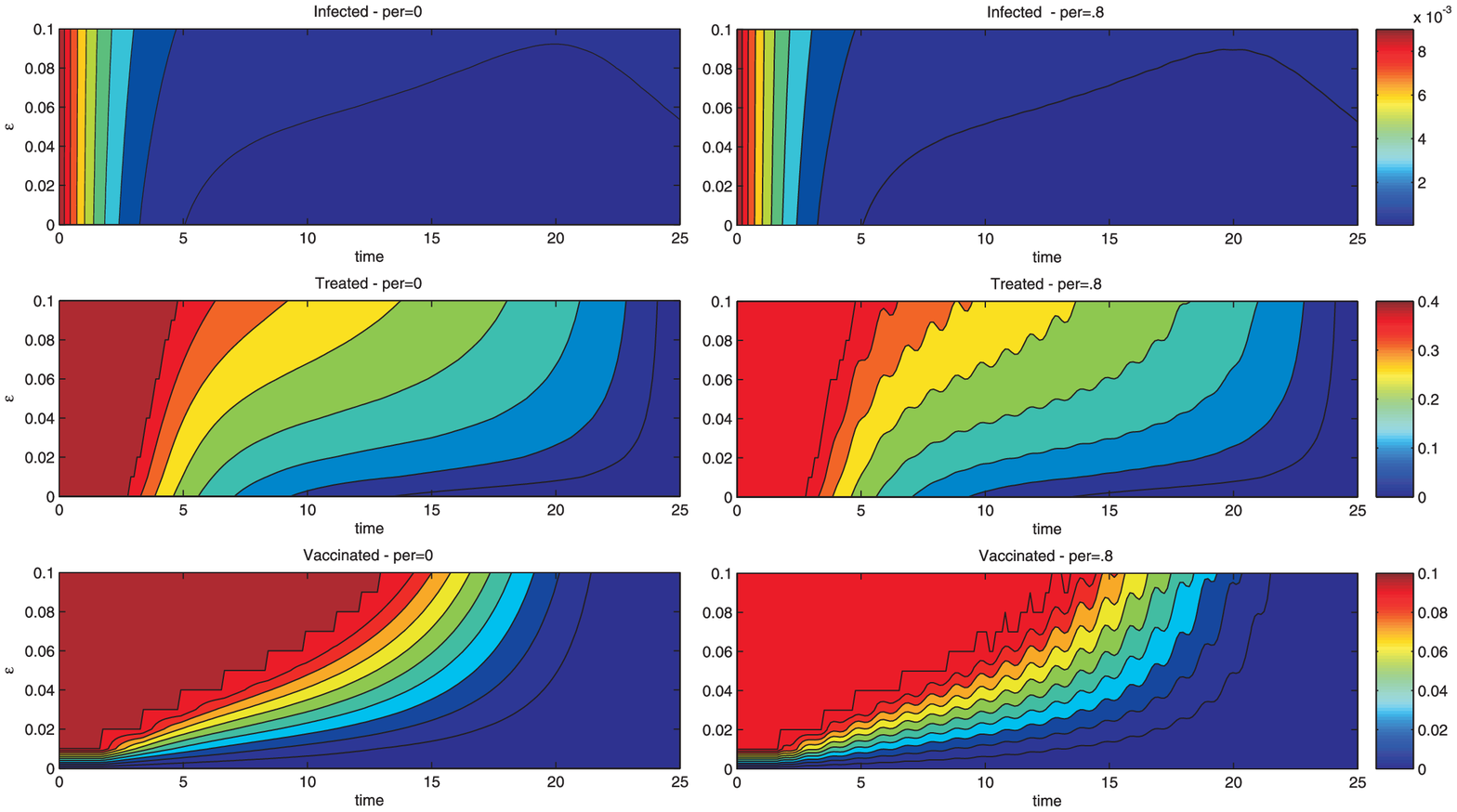}
\caption{\small{Variation of infected individuals $I^*(t)$
and optimal measures of treatment $\mathbbm{T}^*(t)$ and vaccination $\mathbbm{V}^*(t)$,
in both autonomous ($\mathrm{per}=0$) and periodic ($\mathrm{per}=0.8$) cases,
with the infectivity rate $\eps \in [0, 0.1]$.}}
\label{fig_epsilon}
\end{figure}

Finally, in Figure~\ref{fig_eta}, the variation of
the loss of immunity rate $\eta$ is highlighted.
We conclude that the periodicity effect is similar
to the previous considered scenarios. However,
the variation of $\eta$ is the one that less influences
the behavior of the three variables $I^*(t)$,
$\mathbbm{T}^*(t)$ and $\mathbbm{V}^*(t)$.
\begin{figure}[ht]
\centering
\includegraphics[scale=0.442]{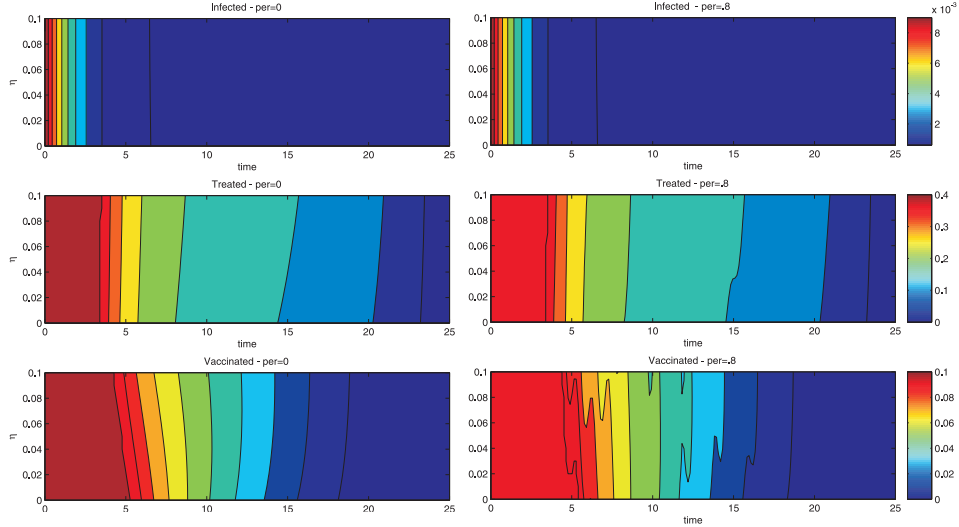}
\caption{\small{Variation of infected individuals $I^*(t)$
and optimal measures of treatment $\mathbbm{T}^*(t)$ and vaccination $\mathbbm{V}^*(t)$,
in both autonomous ($\mathrm{per}=0$) and periodic ($\mathrm{per}=0.8$) cases,
with the loss of immunity rate $\eta \in [0, 0.1]$.}}
\label{fig_eta}
\end{figure}

It is worth mentioning that, in the simulations done
and in the range of parameters considered,
the variation of the periodic parameter $\mathrm{per}$
has a very small effect on the obtained cost functional $\matj$.
Namely, we saw numerically that
$$
\left|\matj\parc{I,\mathbbm{T},\mathbbm{V}}\big|_{\mathrm{per}=v_1}
-\matj\parc{I,\mathbbm{T},\mathbbm{V}}\big|_{\mathrm{per}=v_2}\right|
\leq \mbox{0.000329537},
$$
for  $v_1,v_2\in \{0,0.8\}$.


\section*{Acknowledgments}

The authors are grateful to an anonymous Referee for
valuable comments and suggestions.



\medskip

Received April 2017; revised June 2017.

\medskip


\end{document}